\documentclass{amsart}
\usepackage[all]{xy}

\usepackage[top=3.2cm,bottom=3cm,right=3.2cm,left=3.2cm]{geometry}  

\newtheorem{theorem}{Theorem}
\newtheorem{lemma}{Lemma}
\newtheorem{proposition}{Proposition}
\newtheorem{corollary}{Corollary}

\newtheorem{definition}{Definition}

\theoremstyle{remark}
\newtheorem{remark}{Remark}[section]
\numberwithin{equation}{section}

\newcommand{\C}{{\mathbb C}}       
\newcommand{\R}{{\mathbb R}}       
\newcommand{\D}{{\mathbb D}} 
\newcommand{\N}{{\mathbb{N}}}

\newcommand{\Card}{\text{Card}}
\newcommand{\vdm}{\text{VDM}}

\begin{document}

\title[Multidimensional intertwining Leja sequences and applications]
{Multidimensional intertwining Leja sequences
and applications in bidimensional Lagrange interpolation}

\author{Amadeo Irigoyen}


\email{
\begin{minipage}[t]{5cm}
axirigoyen@gmail.com
\end{minipage}
}


\begin{abstract}

We first give a method to get multidimensional Leja sequences
by considering intertwining sequences from one-dimensional ones. An application is the 
existence of explicit Leja sequences for the
closed unit polydisc.

Next, we deal with some applications in bidimensional Lagrange interpolation
with intertwining Leja sequences.
These results also require an explicit formula for the associated fundamental Lagrange polynomials
with uniform estimates.

\end{abstract}

\maketitle

\tableofcontents

\section{Introduction}

\subsection{Some reminders on the one-dimensional case}

In this paper we deal with multidimensional Leja sequences and some 
estimates for bidimensional Lagrange interpolation.
We remind the reader the expression of the
fundamental Lagrange interpolation polynomial ({\em FLIP}) for the one-dimensional case by
\begin{eqnarray}\label{lagsec}
l_k^{(N+1)}(z) & = &
\prod_{j=0,j\neq k}^{N}
\frac{z-\eta_j}{\eta_k-\eta_j}
\,,\qquad z\in\C,
\quad k=0,\ldots,N
\,,
\end{eqnarray}
where $N\geq0$ and $\eta_0,\ldots,\eta_N$ are all different complex numbers.

The problem of finding {\em good} sets $\left\{\eta_k\right\}_{k\geq0}$ for
Lagrange interpolation
(i.e. for which we can have some control of the associated FLIPs)
is a domain of big interest. 
One of them is called \textit{Fekete set}~(see~\cite{fekete1}):
an $N$-{\em Fekete set} for the compact subset
$K\subset\C$ is a set of $N$ elements
$\zeta_0,\ldots,\zeta_{N-1}\in K$ which
maximize (in modulus) the Vandermonde determinant, i.e.
\begin{eqnarray}\label{defvdmfekete}
\;\;
& &
\left|
\vdm\left(\zeta_0,\ldots,\zeta_{N-1}\right)
\right|
=
\sup_{z_0,\ldots,z_{N-1}\in K}
\left|
\vdm\left(z_0,\ldots,z_{N-1}\right)
\right|
=
\sup_{z_0,\ldots,z_{N-1}\in K}
\prod_{1\leq i<j\leq N-1}\left|z_j-z_i\right|
.
\end{eqnarray}
It is known
that the FLIPs associated with
any $N$-Fekete set are always
bounded by $1$: this can be shown by noticing that
every FLIP can be written for all $z\in\C$ as
\begin{eqnarray}\label{defflip1}
l_{k}^{(N)}(z)
& = &
\frac{
\vdm\left(\zeta_0,\ldots,\zeta_{k-1},\,z\,,\zeta_{k+1},\ldots,\zeta_{N-1}\right)
}{
\vdm\left(\zeta_0,\ldots,\zeta_{N-1}\right)
}
\end{eqnarray}
(where we have replaced $\zeta_k$ with $z$).
It follows by~(\ref{defvdmfekete}) that
$\sup_{z\in K}
\left|
l_{k}^{(N)}(z)
\right|
\;=\;
1$
for all
$k=0,\ldots,N-1$.
Fekete sets are essentially the best ones for Lagrange interpolation
and uniform stability of the associated FLIPs.
Nevertheless, constructing them is generally a hard task.
Therefore, a natural question is if there exist {\em simpler} sets
with (almost) the same property. 
This leads to the following definition.

\begin{definition}\label{defleja}

A {\em Leja sequence $\mathcal{L}$} for a compact set $K\subset\C$ is a sequence
$\left(\eta_0,\eta_1,\ldots,\eta_k,\ldots\right)$ 
that satisfies the following properties: 
$\eta_0
\in
\partial K$
and for all
$k\geq1$,
\begin{eqnarray}\label{defleja1}
\sup_{z\in K}
\left[
\prod_{i=0}^{k-1}
\left|
z-\eta_i
\right|
\right]
& = &
\prod_{i=0}^{k-1}
\left|
\eta_k-\eta_i
\right|
\,.
\end{eqnarray}

For all $N\geq1$, the
{\em $N$-Leja section $\mathcal{L}_N$} of a Leja sequence
$\mathcal{L}$ is the finite sequence given by the first
$N$ points of $\mathcal{L}$.

\end{definition}

These sequences took their name from F. Leja (see~\cite{leja1}) but they were first
considered by A.~Edrei (see \cite{edrei1}, p. 78).
They are not necessarily unique (as the Fekete sets). On the other hand,
by the maximum principle,
all the $\eta_i$'s lie on the boundary $\partial K$.
Moreover, determining
Leja sequences is a $1$-dimensional optimization problem and is inductive 
(unlike any $N$-Fekete set
requires a new research of an $N$-tuple
$\left(\zeta_0,\ldots,\zeta_{N-1}\right)$ for
every $N\geq1$).

In the special case when
$K=\overline{\D}=
\left\{z\in\C,\,|z|\leq1\right\}$
is the closed unit disk, all the Leja sequences are explicit.
One can find in~\cite{bialascalvi} their complete description with proof: if we fix
$\eta_0=1$, we have for all $k\geq0$,
\begin{eqnarray}\label{explicitleja}
\eta_k
\;=\;
\exp
\left(
i\pi\sum_{l=0}^sj_l2^{-l}
\right)
& \mbox{ where } &
k
\;=\;
\sum_{l=0}^sj_l2^l
\,,\;
j_l\in\{0,1\}
\,.
\end{eqnarray}

We recently know that  any $N$-Leja section for the disk
has essentially the same property as any $N$-Fekete set in the meaning that
all the associated FLIPS are uniformly bounded with respect to
$N\geq1$ and $k=0,\ldots,N-1$ (see~\cite[Theorem~1.2]{irigoyen5}).
An immediate consequence is an estimate as 
$O(N)$ 
of the Lebesgue constant 
$\Lambda_{N}\left(\overline{\D}\right)$, $N\geq1$,  of any $N$-Leja section
for the unit disk. 
We remind the reader that the Lebesgue constant is defined for $N\geq1$ by
\begin{eqnarray}\label{defLebcte1}
\Lambda_{N}\left(\overline{\D}\right)
& = &
\sup_{z\in\overline{\D}}
\left(
\sum_{k=0}^{N-1}
\left|
l_k^{(N)}(z)
\right|
\right)
\,.
\end{eqnarray}
J.-P. Calvi and V. M. Phung had conjectured in~\cite{calviphung1} that
$\Lambda_{N}\left(\overline{\D}\right)\leq N$. This conjecture has been confirmed by
M.~Ouna\" ies~(see~\cite{myriam}).

\subsection{Construction of explicit multidimensional Leja sequences}\label{applmultivarintro}

First, we need to define a numeration in $\N^s$ where $s\geq2$. Let be $k\in\N^s$,
i.e. $k=(k_1,\ldots,k_s)$. We set
$|k|=k_1+\cdots+k_s$.\\
For all $k,\,l\in\N^s$, we say that $k\leq l$ iff
\begin{eqnarray}\label{lexicorder}
|k|\;<\;|l|
& \text{or} &
\left\{\begin{array}{l}
|k|\;=\;|l|\\
\text{and}\\ 
k\,\leq\,l
\text{ in the \textit{lexicographic order}\,.}
\end{array}\right.
\end{eqnarray}
We remind the reader that the lexicographic order is defined as follows:
$k< l$ iff $k_1<l_1$ or
$k_1=l_1,\,\ldots,k_t=l_t$ for $t\in[\![1,s-1]\!]$ and $k_{t+1}<l_{t+1}$.
This allows us to define a numeration on $\N^s$: 
\begin{eqnarray}\label{numNs}
\left\{\begin{array}{ccl}
\N^* & \to & \N^s\\
n & \mapsto & k(n)\,=\,\left(k_1(n),\ldots,k_s(n)\right)
\end{array}\right.
& &
\qquad\qquad\qquad\qquad\qquad\qquad\qquad\qquad\qquad
\end{eqnarray}
For example we find for $s=3$:\\
$k(1)=(0,0,0),\,
k(2)=(0,0,1),\,
k(3)=(0,1,0),\,
k(4)=(1,0,0),\,
k(5)=(0,0,2),
\,k(6)=(0,1,1),
\ldots$

Next, let consider the complexe space $\C^s$ with $s\geq2$. Given $s$ complex sequences
$\left(\eta^{(j)}_i\right)_{i\geq0}\subset\C$, $j=1,\ldots,s$, 
we want to define their \textit{intertwining sequence} as a sequence of $\C^s$
(see~\cite{calvi2005}).
On the other hand, we also need to define a numeration on the monomials
$z_1^{k_1}\cdots z_s^{k_s}\in\C\left[z_1,\ldots,z_s\right]=:\C[z]$.

\begin{definition}

The \textit{intertwining sequence} of the sequences 
$\left(\eta^{(j)}_i\right)_{i\geq0}\subset\C$, $j=1,\ldots,s$, 
is the sequence $\eta_{k(n)}\subset\C^s$ defined as:
\begin{eqnarray}\label{numseq}
n\in\N^*\;\mapsto\;
\eta_{k(n)}\,:=\,
\left(\eta^{(1)}_{k_1(n)},\ldots,\eta^{(s)}_{k_s(n)}\right)\,,
\end{eqnarray}
where $k(n)=\left(k_1(n),\ldots,k_s(n)\right)$
is the numeration defined by~(\ref{numNs}).\\
Similarly, we consider the sequence of monomials defined as:\\
\begin{eqnarray}\label{numpol}
n\in\N^*\;\mapsto\;
e_n(z)\,=\,z^{k(n)}\,:=\,
z_1^{k_1(n)}\cdots z_s^{k_s(n)}\,.
\end{eqnarray}

\end{definition}

Since for all $k\in\N^s$, 
$\eta_k=\left(\eta^{(1)}_{k_1},\ldots,\eta^{(s)}_{k_s}\right)$
(resp., $z^k=z_1^{k_1}\cdots z_s^{k_s}$) is reached by a~(unique) integer
$n\in\N^*$, this allows us to consider for all $N\geq1$ the set
$\Omega_N\subset\C^s$ and the space $\mathcal{P}_N\subset\C[z]$:
\begin{eqnarray}\label{defPN}
\Omega_N
\;:=\;
\left\{
\eta_{k(n)}\,,\;n=1,\ldots,N
\right\}
& \text{ and } &
\mathcal{P}_N
\;:=\;
\text{span}
\left\{
e_n\,,\;n=1,\ldots,N
\right\}
\,.
\end{eqnarray}
We always have 
$\dim\mathcal{P}_N=N$.
On the other hand, if all the sequences
$\left(\eta^{(j)}_i\right)_{i\geq0}$, $j=1,\ldots,s$, 
are of pairwise distinct elements (and it will be always the case
in the whole paper), we have
$\Card\left(\Omega_N\right)=N$.
In particular, if we set for all $d\geq0$,
\begin{eqnarray}\label{defNd}
N_d & = &
\frac{(d+1)(d+2)\cdots(d+s)}{s!}
\;=\;
\dbinom{s+d}{s}
\,,
\end{eqnarray}
then 
\begin{eqnarray}\label{defOmegaNPNple}
\Omega_{N_d}\;=\;\left\{\eta_k\,,\;|k|\leq d\right\}\,,
& \text{ and } &
\mathcal{P}_{N_d}\;=\;\C_d[z_1,\ldots,z_s]\;=:\;\C_d[z]
\end{eqnarray}
the space of polynomials whose total degree is at most $d$.

On the other hand,
definition~(\ref{numpol}) allows us to define the generalized
Vandermonde determinant of any $N$-tuple 
$(H_1,\ldots,H_N)\in\left(\C^s\right)^N$:
\begin{eqnarray}\label{defMultiVDM}
\vdm\left(H_1,\ldots,H_N\right)
& = &
\det\left[e_i\left(H_j\right)\right]_{1\leq i, j\leq N}
\end{eqnarray}
(we indeed get back the usual Vandermonde determinant for $s=1$
since~(\ref{numNs}) and~(\ref{numpol}) yield: $\forall\,n\geq1$,
$e_n(z_1)=z_1^{n-1}$).
Since the sequences $\left(\eta^{(j)}_i\right)_{i\geq0}$
are of pairwise distinct elements, we have:
\begin{eqnarray*}
\forall\,N\geq1\,,\qquad
\vdm\left(\Omega_N\right)
\;:=\;
\det\left[e_{k(i)}\left(\eta_{k(j)}\right)\right]_{1\leq i,j\leq N}
\;\neq\;0
\,.
\end{eqnarray*}
This means that the set $\Omega_N$ is \textit{unisolvent} for the space $\mathcal{P}_N$
and an equivalent condition is the following one:
for all $P\in\mathcal{P}_N$, if $P(z)=0$ for all
$z\in\Omega_N$, then $P\equiv0$. 
This extends in the special case of $\Omega_N\subset\C^s$ the classical idea of 
Biermann in~\cite{bierman1903}
(see~\cite{calvi2005} as well for a general version with
{\em block unisolvent arrays}).
We can also give a direct proof of this result 
(Section~\ref{applmultivar}, Remark~\ref{OmegaNunisolvent}).

A consequence is the following equivalent property:
for every function $f$ defined on
$\Omega_N$, there exists a unique polynomial $P\in\mathcal{P}_N$
such that $P(z)=f(z)$
for all $z\in\Omega_N$. 
This element $P$ is the
\textit{Lagrange interpolation polynomial $L_{\Omega_N}[f]$ of $f$ with respect to $\Omega_N$}
(see~(\ref{defmultilagpol}) below).
In addition, $L_{\Omega_N}[f]$ can be given as follows:
since $\vdm\left(\Omega_N\right)\neq0$, 
we can write in this fashion
the (generalized) fundamental Lagrange polynomials (FLIPs)
in the form~(\ref{defflip1})
(where $H_n=\eta_{k(n)}$, $n=1,\ldots,N$), i.e.
\begin{eqnarray}\label{defmultiflip}
\qquad
& &
l_{H_n}^{(N)}(z)
\;=\;
\dfrac{\vdm\left(H_1,H_2,\ldots,H_{n-1},\;z\;,H_{n+1},\ldots,H_N\right)
}{
\vdm\left(\Omega_N\right)}
\,,\quad
\forall\,z\in\C^s
\end{eqnarray}
thus
\begin{eqnarray}\label{defmultilagpol}
L_{\Omega_N}[f](z)
& = &
\sum_{n=1}^N
f\left(H_n\right)
\,
l_{H_n}^{(N)}(z)
\,.
\end{eqnarray}
As it has been pointed out (e.g., see~\cite[p. 54]{bloomboscalvilev2012}), there is no cancellation 
in~(\ref{defmultiflip})
so that the formulas cannot be simplified as in~(\ref{lagsec}). 
Nevertheless, Proposition~\ref{explicitlagrangepol}
in Section~\ref{appl2var} gives for the special bidimensional case 
an explicit formula of $l_{H_n}^{(N)}$, $\forall\,n=1,\ldots,N$.

Now definition~(\ref{defMultiVDM}) allows us to extend from
the one-dimensional case~(\ref{defleja1}) 
the generalized definition of 
a (multidimensional) Leja sequence in $\C^s$.
By analogy, given a compact set $K\subset\C^s$, we can choose $H_1\in\partial K$
and, if we have constructed
$H_1,\ldots,H_{N}$, we choose for $H_{N+1}$ any point (not necessarily unique) that satisfies
\begin{eqnarray*}
\left|
\vdm\left(H_1,H_2,\ldots,H_N,H_{N+1}\right)
\right|
& = &
\sup_{z\in K}
\left|
\vdm\left(H_1,H_2,\ldots,H_N,z\right)
\right|
\,.
\end{eqnarray*}
As for the Fekete sets, it is hard to construct examples of explicit 
multidimensional Leja sequences (even in one dimension, except for special compact sets
as the disk).
Here we present a way to get explicit multidimensional Leja sequences for the product of compact sets
(provided we already know their partial one-dimensional Leja sequences).

\begin{theorem}\label{explicitlejasequ}

For all $j=1,\ldots,s$, let $K_j\subset\C$ be a compact set and
$\left(\eta^{(j)}_i\right)_{i\geq0}\subset K_j$ any sequence
(with $\eta^{(j)}_0\in\partial K_j$),
and let $\left(H_n\right)_{n\geq1}=\left(\eta_{k(n)}\right)_{n\geq1}$
be their intertwining sequence. TFAE:
\begin{enumerate}
\item[$(i)$]
$\left(H_n\right)_{n\geq1}$ is a Leja sequence for $K_1\times\cdots\times K_s$;

\item[$(ii)$]
$\forall\,j=1,\ldots,s$, $\left(\eta^{(j)}_i\right)_{i\geq0}$
is a Leja sequence for $K_j$.

\end{enumerate}

\end{theorem}

First, a natural question is if all Leja sequence can be written as an intertwining one from
one-dimensional Leja sequences.
The answer is negative and we give 
a counterexample in 
Subsection~\ref{Lejanonintertwining}~(Proposition~~\ref{ctrexLejanonint}).

Next, as an immediate application of the above theorem, 
we can give some explicit multidimensional Leja sequences for $\overline{\D}^s$
(as well as for any closed polydisc of $\C^s$ by dilatation and translation).

\begin{corollary}\label{explicitbileja}

For all $j=1,\ldots,s$, let $\left(\eta^{(j)}_i\right)_{i\geq0}$ be
defined as in~(\ref{explicitleja}). Then their intertwining sequence is a
Leja sequence for the unit polydisc $\overline{\D}^s$.

\end{corollary}

Theorem~\ref{explicitlejasequ} is an application of Proposition~\ref{explicitVDM}
(Section~\ref{applmultivar}) that gives an inductive formula for the Vandermonde
determinant. Another application is the formula of M.~Schiffer and J.~Siciak
for the two-variable Vandermonde determinant
(see~\cite{schiffersiciak1962}, or formula~(4.8.2) from~\cite{bloomboschrislev1992}).

\begin{corollary}\label{redemschifsic}

For all $d\geq0$,
\begin{eqnarray*}
\vdm
\left(
\left(\eta_0,\theta_0\right),\ldots,\left(\eta_{d-1},\theta_0\right),
\left(\eta_0,\theta_d\right),\ldots,\left(\eta_d,\theta_0\right)
\right)
\,=\,
\prod_{j=1}^d
\left[
\vdm\left(\eta_0,\ldots,\eta_j\right)
\times
\vdm\left(\theta_0,\ldots,\theta_j\right)
\right]
.
\end{eqnarray*}

\end{corollary}

\subsection{Some estimates in bidimensional Lagrange interpolation}\label{appl2varintro}

In this part we deal with the special case of $\C^2$.
We first remind the reader the lexicographic order~(\ref{lexicorder})
and~(\ref{numNs}) for $\N^2$ (i.e.
$n\mapsto\left(k(n),l(n)\right)$)
that yields numerations~(\ref{numseq}) and~(\ref{numpol})
(where $\left(\eta_i\right)_{i\geq0}$ and $\left(\theta_j\right)_{j\geq0}$ are two 
sequences of pairwise distinct elements):
\begin{eqnarray}\label{defintertwin}
\left\{H_n\;=\;\left(\eta_{k(n)},\theta_{l(n)}\right)\,,\;n\geq1\right\}
& \text{ and } &
\left\{e_n(z,w)\;=\;z^{k(n)}w^{l(n)}\,,\;n\geq1\right\}
\,.
\end{eqnarray}
We also remind the reader the space
$\C_d[z,w]$ of complex polynomials of total degree at most $d$, and
\begin{eqnarray}\label{Nn}
N_d
\;=\;
\dim\C_d[z,w]
\;=\;
\dfrac{(d+1)(d+2)}{2}
\end{eqnarray}
(with the convention that $N_{-1}=0$ and $\C_{-1}[z,w]=\{0\}$).
Since we deal with the bidimensional case, we can 
consider for all $N\geq1$ the unique integers $d\geq0$ and $m$ with $0\leq m\leq d$, such that
\begin{eqnarray}\label{defnm}
N_{d-1}
\;<\;
N
\;\leq\;
N_d
& \mbox{ and } &
m
\;=\;
N-N_{d-1}-1
\,.
\end{eqnarray}
We also remind the reader the intertwining section
\begin{eqnarray}\label{defOmegaN2}
\quad
& &
\Omega_N
\;:=\;
\left\{H_n\,,\;1\leq n\leq N\right\}
\;=\;
\left\{
\left(\eta_0,\theta_0\right),\left(\eta_0,\theta_1\right),\left(\eta_1,\theta_0\right),
\ldots,
\left(\eta_0,\theta_d\right),
\ldots,\left(\eta_m,\theta_{d-m}\right)
\right\}
\end{eqnarray}
and the space
\begin{eqnarray}\label{defPN2}
\mathcal{P}_N
& = &
\text{span}\left\{e_n\,,\;1\leq n\leq N\right\}
\,.
\end{eqnarray}
One has $\C_{d-1}[z,w]\subset\mathcal{P}_N\subset\C_{d}[z,w]$, 
and $\mathcal{P}_N$ is spanned by $\C_{d-1}[z,w]$
and the monomials $w^d,\,zw^{d-1},\ldots,z^{m}w^{d-m}$.
In particular, $\mathcal{P}_{N_d}=\C_d[z,w]$.
On the other hand, we have
by~(\ref{defnm}): $\dim\mathcal{P}_N=N_{d-1}+m+1=N$.\\
Thus we can remind the reader 
the fundamental Lagrange polynomials (FLIPs) that are well-defined
since $\Omega_N$ is unisolvent for $\mathcal{P}_N$: 
if $H_n=\left(\eta_{k(n)},\theta_{l(n)}\right)$, $1\leq n\leq N$, then
\begin{eqnarray*}
\qquad
& &
l_{H_n}^{(N)}(z,w)
\;=\;
\dfrac{\vdm\left(H_1,H_2,\ldots,H_{n-1},\,(z,w)\,,H_{n+1},\ldots,H_N\right)
}{
\vdm\left(\Omega_N\right)}
\,,
\quad
\forall\,(z,w)\in\C^2
\,.
\end{eqnarray*}
The bidimensional Lagrange polynomial of any function
$f$ defined on $\Omega_N$ is given by:
\begin{eqnarray}\label{deflagpold}
L_{\Omega_N}[f](z,w)
& = &
\sum_{n=1}^N
f\left(H_n\right)
\,
l_{H_n}^{(N)}(z,w)
\,.
\end{eqnarray}
Finally, the Lebesgue constant of the set
$\Omega_N=\left\{H_1,\ldots,H_N\right\}$ (with respect to a given compact set
$K\supset\Omega_N$) is defined similarly as in the one-dimensional 
case~(\ref{defLebcte1}), i.e.:
\begin{eqnarray*}
\Lambda_N\left(K,\Omega_N\right)
& = &
\sup_{(z,w)\in K}
\left[
\sum_{n=1}^N
\left|
l_{H_n}^{(N)}(z,w)
\right|
\right]
\,.
\end{eqnarray*}
We then have the following result.

\begin{theorem}\label{lebcteintertwin}

Let $\left(\eta_i\right)_{i\geq0}$ and $\left(\theta_j\right)_{j\geq0}$
be Leja sequences for the unit disk (with $\left|\eta_0\right|=\left|\theta_0\right|=1$),
and let us consider the intertwining sequence $\left(H_n\right)_{n\geq1}$ defined
as in~(\ref{defintertwin}). We have for all $N\geq1$:
\begin{eqnarray}\label{lebcteintertwinbidisc}
\Lambda_N\left(\overline{\D}^2,\Omega_N\right)
& = &
O\left(N^{3/2}\right)
\,.
\end{eqnarray}

If $K_1$ (resp., $K_2$) is a compact set whose boundary is an Alper-smooth Jordan curve,
$\Phi_1$ (resp., $\Phi_2$) its associated conformal mapping, let us consider
the sequence $\left(\Phi\left(H_n\right)\right)_ {n\geq1}$ defined by the intertwining sequence of
$\left(\Phi_1\left(\eta_i\right)\right)_{i\geq0}$
and
$\left(\Phi_2\left(\theta_j\right)\right)_{j\geq0}$.
Then
\begin{eqnarray}\label{lebcteintertwincompact}
\Lambda_N\left(K_1\times K_2,\Phi\left(\Omega_N\right)\right)
& = &
O\left(N^{2A/\ln(2)+3/2}\right)
\,,
\end{eqnarray}
where $A$ is a positive constant depending only on $K_1\times K_2$.

\end{theorem}

As a consequence, this result gives a partial answer to question~(6) in~\cite{bloomboscalvilev2012}.

We remind the reader that a compact
set whose boundary is an Alper-smooth Jordan curve, is
a special class of compact sets: for example, twice continuously 
differentiable Jordan curves are Alper-smooth. 
Here $\Phi_1$ and $\Phi_2$ denote the respective conformal mappings from
$\overline{\C}\setminus\D$ onto $\overline{\C}\setminus K_1$ and $\overline{\C}\setminus K_2$.
This theorem is an immediate application of the
following result that is a uniform estimate
for the bidimensional FLIPs associated with any intertwining Leja
sequence,
and that is an application of Proposition~\ref{explicitlagrangepol} (Section~\ref{appl2var})
that gives explicit formulas for $l_{H_n}^{(N)}$, $\forall\,n=1,\ldots,N$.

\begin{proposition}\label{unifestimintertwining}

Under the hypotheses of the first part of Theorem~\ref{lebcteintertwin}, one has for all $N\geq1$
and all $p,\,q\geq0$ such that
$\left(\eta_p,\theta_q\right)\in\Omega_{N}$
(we have $p+q\leq d$ by~(\ref{defnm})),
\begin{eqnarray*}
\sup_{(z,w)\in\overline{\D}^2}
\left|
l_{(\eta_p,\theta_q)}^{(N)}(z,w)
\right|
\;\leq\;
2\left(d-p-q+1\right)
\pi^2\exp(6\pi)
\;=\;
O\left(N^{1/2}\right)
\,.
\end{eqnarray*}

Similarly, under the hypotheses of the second part of Theorem~\ref{lebcteintertwin}, one has for all $N\geq1$
and all $p,\,q\geq0$ such that
$\left(\eta_p,\theta_q\right)\in\Omega_{N}$,
\begin{eqnarray*}
\sup_{(z,w)\in K_1\times K_2}
\left|
l_{(\Phi_1(\eta_p),\Phi_2(\theta_q))}^{(N)}(z,w)
\right|
\;\leq\;
M(d+4)^{4A/\ln(2)}
\left(d-p-q+1\right)
\;=\;
O\left(N^{2A/\ln(2)+1/2}\right)
,
\end{eqnarray*}
where $M$ and $A$ are positive constants depending only on $K_1\times K_2$.

\end{proposition}

Although there are results which give polynomial estimates for $\C^s$
in~\cite{calviphung1}, Theorem~\ref{lebcteintertwin} yields
here precise exponents for $N$. In particular, it gives examples of
sequences that satisfy\\
$\lim_{d\rightarrow+\infty}\left[\Lambda_{N_d}\left(K,\Omega_{N_d}\right)\right]^{1/d}=1$,
where $K$ is any compact set mentioned in
Theorem~\ref{lebcteintertwin}.
In addition, there is an 
improvement of some results from~\cite{calviphung1} for $s=2$: on the one hand, it is proved
in~\cite[p.~621]{calviphung1} that
$\Lambda_{N_d}\left(\overline{\D}^2,\Omega_N\right)
=
O\left(
d^{(2^2+7\times2-6)/2}(\ln d)^2
\right)
=
O\left(d^{6}(\ln d)^2\right)$;
on the other hand, 
an application of~(\ref{lebcteintertwinbidisc}) with $N=N_d=O\left(d^2\right)$ (by~(\ref{Nn})), gives
$\Lambda_{N_d}\left(\overline{\D}^2,\Omega_N\right)
=
O\left(N_d^{3/2}\right)
=
O\left(d^3\right)$.

Since any compact set described in 
Theorem~\ref{lebcteintertwin} 
is nonpluripolar, polynomially convex
and $L$-regular, inequality~(\ref{lebcteintertwincompact}) has well known 
consequences~(see~\cite[Theorem~4.1]{bloomboschrislev1992}):

\begin{itemize}

\item

$\lim_{d\rightarrow+\infty}\left|\vdm\left(H_1,H_2,\ldots,H_{N_d}\right)\right|^{1/l_d}=D(K)$,
where $l_d=\sum_ {j=1}^dj\left(N_j-N_{j-1}\right)$ and 
$D(K)$ is the transfinite diameter of $K$ (see~\cite{zaharjuta1975} for its general definition 
and existence);

\item

by~\cite{bermanboucksomwitt2011}, the empirical measures satisfy
$\lim_{d\rightarrow+\infty}\dfrac{1}{N_d}\sum_{j=1}^{N_d}\delta_{H_j}=c\mu_K$
weak-*,
where $\delta_{H_j}$ is the unit Dirac measure on $H_j$, $c$ is some universal constant and
$\mu_K$ is the equilibrium measure of $K$ (for the definition and existence of $\mu_K$,
see~\cite[Chapter~1]{safftotik} where $K\subset\C$, and~\cite[Chapter~5]{klimek} where $K\subset\C^d$);

\item

the Lagrange polynomial $L_{\Omega_{N_d}}[f]$ converges to $f$ as $d\rightarrow+\infty$, 
uniformly on $K$ and for each $f$ holomorphic on a
neighborhood of $K$.

\end{itemize}

Another consequence of the above theorem is a usual application of an estimate of the
Lebesgue constant
to approximation properties of Lagrange polynomial interpolation.
Theorem~7.2
from~\cite{renwang} provides a generalized version of Jackson's Theorem
(see~\cite{jackson}).
Given $m\in\N$ and $\gamma$ with $0<\gamma\leq1$, we consider the space
$Lip_{\gamma}^m\left(\overline{\D}^2\right)$ of the functions 
$f\in\mathcal {O}\left(\D^2\right)\bigcap C\left(\overline{\D}^2\right)$
which satisfy for all $|\alpha|=\alpha_1+\alpha_2\leq m$,
$\left(
\dfrac{\partial^{\alpha}f}{\partial z^{\alpha}}
\right)
\left(e^{ih}\zeta\right)
-
\left(
\dfrac{\partial^{\alpha}f}{\partial z^{\alpha}}
\right)
(\zeta)
=
O\left(|h|^{\gamma}\right)$,
for all
$\zeta\in\mathbb{T}^2$ (the unit torus in $\C^2$) and $h\in\R$.
We then have the following result.

\begin{corollary}\label{appljackson}

Let be $m\geq3$, $\gamma$ with $0<\gamma\leq1$, and
the intertwining sequence $\left(H_n\right)_{n\geq1}$
of any two Leja sequences for the unit disk.
Then for all $f\in Lip_{\gamma}^m\left(\overline{\D}^2\right)$,
\begin{eqnarray*}
\sup_{(z,w)\in\overline{\D}^2}
\left|f(z,w)-L_{\Omega_N}[f](z,w)\right|
& = &
O\left(\frac{1}{N^{(m+\gamma-3)/2}}\right)
\,,
\end{eqnarray*}
where $L_{\Omega_N}[f]$ is the bidimensional Lagrange polynomial of $f$ defined in~(\ref{deflagpold}).

\end{corollary}

\section*{Acknowledgments}

I would like to thank J. Ortega-Cerd\`a for having introduced me this nice problem and
for all our interesting discussions about this work.


\section{On the construction of multidimensional Leja sequences}\label{applmultivar}

In all the following, we will take the convention that any empty sum will be
$0$. Similarly, any empty product will be $1$.

\subsection{A preliminary formula for the Vandermonde determinant of intertwining sequences}

We first deal  with an important result that is a formula
for the Vandermonde determinant for any intertwining sequence. It 
can also be linked to Proposition~2.5 from~\cite{calvi2005}.

\begin{proposition}\label{explicitVDM}

Let $\left(H_n \right)_{n\geq1}$ be the intertwining sequence of any 
sequences $\left(\eta^{(1)}_{i_1}\right)_{i_1\geq0},\ldots$ and $\left(\eta^{(s)}_{i_s}\right)_{i_s\geq0}$
(where for all $j=1,\ldots,s$, $\left(\eta^{(j)}_i\right)_{i\geq0}$ is a sequence
of pairwise distinct elements).
We have for all $N\geq1$ and $z=(z_1,\ldots,z_s)\in\C^s$:
\begin{eqnarray*}
\vdm
\left(
H_1,H_2,\ldots,H_{N},z
\right)
& = &
P_N(z)
\times
\vdm
\left(
H_1,H_2,\ldots,H_{N}
\right)
\,,
\end{eqnarray*}
where $P_N\in\C[z]=\C[z_1,\ldots,z_s]$.

In addition, let denote $H_N=\left(\eta^{(1)}_{k_1},\ldots,\eta^{(s)}_{k_s}\right)$
where $k(N)=\left(k_1,\ldots,k_s\right)$:
\begin{enumerate}
\item[$(i)$] 
if $N=N_d=\dbinom{d+s}{s}$ with $d\geq0$ (see~(\ref{defNd})), then 
$k(N)=\left(d,0,\ldots,0\right)$ and 
\begin{eqnarray}\label{explicitVDMi}
P_N(z) & = &
\prod_{i_s=0}^d\left(z_s-\eta^{(s)}_{i_s}\right)
\,;
\end{eqnarray}

\item[$(ii)$] otherwise $N_d<N<N_{d+1}$ and
$k(N)=\left(k_1,\ldots,k_m,0,\ldots,0\right)$ where 
$2\leq m\leq s$ and $k_m\geq1$, then 
\begin{eqnarray}\label{explicitVDMii}
P_N(z) & = &
\prod_{j=1}^{m-2}
\left[
\prod_{i_j=0}^{k_j-1}\left(z_j-\eta^{(j)}_{i_j}\right)
\right]
\times
\prod_{i_{m-1}=0}^{k_{m-1}}\left(z_{m-1}-\eta^{(m-1)}_{i_{m-1}}\right)
\times
\prod_{i_s=0}^{k_m-2}\left(z_s-\eta^{(s)}_{i_s}\right)
\,.
\end{eqnarray}

\end{enumerate}

\end{proposition}

\begin{proof}

We have for all $N\geq1$ and all $z\in\C^s$ (the $e_n$'s being defined by~(\ref{numpol})),
\begin{eqnarray*}
\vdm
\left(
H_1,H_2,\ldots,H_{N},z
\right)
\,=\,
\left|
\begin{array}{ccccc}
e_1\left(H_1\right) & e_1\left(H_2\right) & \cdots & e_1\left(H_N\right) & e_{1}(z)
\\
e_2\left(H_1\right) & e_2\left(H_2\right) & \cdots & e_2\left(H_N\right) & e_{2}(z)
\\
\vdots & \vdots & \ddots & \vdots & \vdots
\\
e_N\left(H_1\right) & e_N\left(H_2\right) & \cdots & e_N\left(H_N\right) & e_{N}(z)
\\
e_{N+1}\left(H_1\right) & e_{N+1}\left(H_2\right) & \cdots & e_{N+1}\left(H_N\right) & e_{{N+1}}(z)
\end{array}
\right|
.
\end{eqnarray*}
Let us consider the Lagrange polynomial of $e_{N+1}$ defined by
$L_{\Omega_N}\left[e_{N+1}\right]=
\sum_{n=1}^Ne_{N+1}\left(H_n\right)l_{H_n}^{(N)}$ 
where $\Omega_N=\left\{H_1,\ldots,H_N\right\}$,
and whose
existence is guaranteed because $\Omega_N$ is unisolvent.
Since $L_{\Omega_N}\left[e_{N+1}\right]$ is spanned by the
$e_n$'s, $n=1,\ldots,N$, one also has
\begin{eqnarray}\label{explicitVDMaux1}
L_{\Omega_N}\left[e_{N+1}\right](z)
& = &
\sum_{n=1}^N\alpha_ne_n(z)
\,,\qquad
\alpha_n\in\C,\,\;\forall\,n=1,\ldots,N.
\end{eqnarray}
The above determinant is not changed if we replace the last row $L_{N+1}$
with $L_{N+1}-\sum_{n=1}^N\alpha_nL_n$. It follows by~(\ref{explicitVDMaux1})
that $e_{N+1}\left(H_j\right)$ becomes for
all $j=1,\ldots,N$, 
\begin{eqnarray*}
e_{N+1}\left(H_j\right)-\sum_{n=1}^N\alpha_ne_n\left(H_j\right)
\;=\;
e_{N+1}\left(H_j\right)-L_{\Omega_N}\left[e_{N+1}\right]\left(H_j\right)
\;=\;
0
\,.
\end{eqnarray*}
On the other hand, the element $e_{N+1}(z)$ becomes
$e_{N+1}(z)-\sum_{n=1}^N\alpha_ne_n(z)=e_{N+1}(z)-L_{\Omega_N}\left[e_{N+1}\right](z)$.
Thus
\begin{eqnarray*}
\vdm
\left(
H_1,H_2,\ldots,H_{N},z
\right)
& = &
\;\;\;\;\;\;\;\;\;\;\;\;\;\;\;\;\;\;\;\;\;\;\;\;\;\;\;\;\;\;\;\;\;\;\;\;\;\;\;\;\;\;\;\;
\;\;\;\;\;\;\;\;\;\;\;\;\;\;\;\;\;\;\;\;\;\;\;\;\;\;\;\;\;\;\;\;\;\;\;\;\;\;\;\;\;\;\;\;
\;\;\;\;\;\;\;\;\;\;\;\;\;\;\;\;\;\;\;\;\;\;\;\;\;\;\;\;\;\;\;\;\;\;\;\;\;\;\;\;\;\;\;\;
\end{eqnarray*}
\begin{eqnarray}\nonumber
\;\;\;\;\;\;\;\;\;\;\;\;\;\;\;\;\;\;\;\;\;\;\;\;
& = &
\left|
\begin{array}{ccccc}
e_1\left(H_1\right) & e_1\left(H_2\right) & \cdots & e_1\left(H_N\right) & e_{1}(z)
\\
e_2\left(H_1\right) & e_2\left(H_2\right) & \cdots & e_2\left(H_N\right) & e_{2}(z)
\\
\vdots & \vdots & \ddots & \vdots & \vdots
\\
e_N\left(H_1\right) & e_N\left(H_2\right) & \cdots & e_N\left(H_N\right) & e_{N}(z)
\\
0 & 0 & \cdots & 0 & e_{{N+1}}(z)-L_{\Omega_N}\left[e_{N+1}\right](z)
\end{array}
\right|
\\\nonumber
\\\nonumber
& = &
\left[
e_{{N+1}}(z)-L_{\Omega_N}\left[e_{N+1}\right](z)
\right]
\times
\vdm
\left(
H_1,H_2,\ldots,H_{N}
\right)
\,.
\end{eqnarray}

We begin with the first case~$(i)$. Since $N=N_d$, we have
$\Omega_N=\left\{\left(\eta^{(1)}_{k_1},\ldots,\eta^{(s)}_{k_s}\right)\,,\;|k|\leq d\right\}$
and 
$\mathcal{P}_N=\C_d[z]$.
On the other hand, by the definition of numeration~(\ref{numNs}),
$k(N)=(d,0,\ldots,0)$ is the biggest $n$-tuple $k$ such that $|k|=d$. It follows that
$k(N+1)=(0,\ldots,0,d+1)$ then:\\
$\left\{\begin{array}{rcl}
H_N=\eta_{k(N)}
& = &
\left(\eta^{(1)}_d,\eta^{(2)}_0,\ldots,\eta^{(s)}_0\right)\\
e_N(z) & = &
z_1^d
\end{array}\right.$
\; and \;\;
$\left\{\begin{array}{rcl}
H_{N+1}=\eta_{k(N+1)}
& = &
\left(\eta^{(1)}_0,\ldots,\eta^{(s-1)}_0,\eta^{(s)}_{d+1}\right)\\
e_{N+1}(z) & = &
z_s^{d+1}\,.
\end{array}\right.$\\
The proof of~(\ref{explicitVDMi}) is equivalent to:
\begin{eqnarray}\label{explicitVDMiaux}
\forall\,z\in\C^s\,,
& &
L_{\Omega_N}[e_{N+1}](z)
\;=\;
e_{N+1}(z)-\prod_{i_s=0}^d\left(z_s-\eta^{(s)}_{i_s}\right)
\,.
\end{eqnarray}
On the one hand, 
\begin{eqnarray*}
e_{N+1}(z)-\prod_{i_s=0}^d\left(z_s-\eta^{(s)}_{i_s}\right)
\;=\;
z_s^{d+1}-\left(z_s^{d+1}+Q_d(z_s)\right)
\;=\;
-Q_d(z_s)\,,
\end{eqnarray*}
where $Q_d\in\C_d[z_s]\subset\C_d[z]=\mathcal{P}_N$
(as well as $L_{\Omega_N}[e_{N+1}](z)$ by definition of the Lagrange
polynomial constructed from $\Omega_N$).
\\
On the other hand, we claim that~(\ref{explicitVDMiaux}) is valid on the subset
$\Omega_N$. Indeed, let be 
$\eta_l=\left(\eta^{(1)}_{l_1},\ldots,\eta^{(s)}_{l_s}\right)\in\Omega_N$, then $|l|\leq d$ and
$l_s\leq|l|\leq d$. Thus
\begin{eqnarray*}
e_{N+1}\left(\eta_l\right)-\prod_{i_s=0}^d\left(\eta^{(s)}_{l_s}-\eta^{(s)}_{i_s}\right)
\;=\;
\left(\eta^{(s)}_{l_s}\right)^{d+1}-0
\;=\;
L_{\Omega_N}[e_{N+1}]\left(\eta_l\right)
\,.
\end{eqnarray*}
Finally, both polynomials belong to $\mathcal{P}_N$ and coincide on the subset $\Omega_N$
that is unisolvent for $\mathcal{P}_N$. This proves equality~(\ref{explicitVDMiaux}).

\medskip

Now we deal with second case~$(ii)$. Since $N_d<N<N_{d+1}$, we have 
$|k(N)|=d+1$.
Either 
$m\leq s-1$ or $m=s$.\\
$(a)$ If $m\leq s-1$, then
$k(N)=(k_1,\ldots,k_m,0,\ldots,0)$ with
$k_1+\cdots+k_m=|k|=d+1$. Then
$k(N+1)=(k_1,\ldots,k_{m-2},k_{m-1}+1,0,\ldots,0,k_m-1)$
(since $2\leq m\leq s-1$ and $k_m\geq1$) and we get
\begin{eqnarray*}
\left\{\begin{array}{rcl}
H_{N+1}\;=\;
\eta_{k(N+1)}
& = &\left(\eta^{(1)}_{k_1},\ldots,\eta^{(m-2)}_{k_{m-2}},\eta^{(m-1)}_{k_{m-1}+1},
\eta^{(m)}_0,\ldots,\eta^{(s-1)}_{0},\eta^{(s)}_{k_m-1}\right)
\\
e_{N+1}(z) & = &
\left(\prod_{j=1}^{m-2}z_j^{k_j}\right)
\times
z_{m-1}^{k_{m-1}+1}z_s^{k_m-1}
\,.
\end{array}\right.
\end{eqnarray*}
Once again, we want to prove that 
\begin{eqnarray}\label{explicitVDMiiaux}
\forall\,z\in\C^s\,,
& &
L_{\Omega_N}[e_{N+1}](z)
\;=\;
e_{N+1}(z)-P_N(z)
\,.
\end{eqnarray}
(where $P_N$ is defined by~(\ref{explicitVDMii})).\\
On the one hand, we have
\begin{eqnarray*}
P_N(z)=\prod_{j=1}^{m-2}\left(z_j^{k_j}+Q_j(z_j)\right)
\times\left(z_{m-1}^{k_{m-1}+1}+Q_{m-1}(z_{m-1})\right)
\times
\left(z_s^{k_m-1}+Q_s(z_s)\right)
\end{eqnarray*}
where for all $j=1,\ldots,m-2$ (resp., $j=m-1$, $s$),
$Q_j\in\C_{k_j-1}[z_j]$ (resp., $Q_{m-1}\in\C_{k_{m-1}}[z_{m-1}]$,
$Q_s\in\C_{k_{m}-2}[z_s]$).
Notice that if $k_j=0$ (resp., $k_m=1$), then $Q_j=0$ (resp., $Q_{s}=0$)
whose degree is at most $-1$.
It follows that, after expanding $P_N(z)$, we get
\begin{eqnarray*}
e_{N+1}(z)-P_N(z)
& = &
0-\sum_{u=1}^M\left(\prod_{j=1}^{m-1}R_{u,j}(z_j)\right)\times R_{u,s}(z_s)
\,,
\end{eqnarray*}
where $M\geq1$, $R_{u,j}(z_j)=z_j^{k_j}$ or $Q_j(z_j)$ for all $j=1,\ldots,m-2$
(resp., $R_{u,m-1}(z_{m-1})=z_{m-1}^{k_{m-1}+1}$ or $Q_{m-1}(z_{m-1})$,
$R_s(z_s)=z_s^{k_m-1}$ or $Q_s(z_s)$).
Moreover, for all $u=1,\ldots,M$, there exists (at least) $j_u=1,\ldots,m-2$
(resp., $j_u=m-1$, $s$) such that $R_{u,j_u}(z_{j_u})=Q_{j_u}(z_{j_u})$
(resp., $R_{u,m-1}(z_{m-1})=Q_{m-1}(z_{m-1})$, $R_{u,s}(z_s)=Q_s(z_s)$).
In any case, the total degree of each product is
\begin{eqnarray*}
\deg\left[\left(\prod_{j=1}^{m-1}R_{u,j}(z_j)\right)\times R_{u,s}(z_s)\right]
\;\leq\;
\sum_{j=1}^{m-2}k_j+(k_{m-1}+1)+(k_m-1)-1
\;=\;
|k(N)|-1\;=\;d,
\end{eqnarray*}
thus
$e_{N+1}(z)-P_N(z)\in\C_d[z]\subset\mathcal{P}_N$
(as well as $L_{\Omega_N}[e_{N+1}](z)$).

On the other hand, in order to
check the validity of~(\ref{explicitVDMiiaux}) on the subset $\Omega_N$,
it is sufficient to prove that 
$P_N\left(\eta_l\right)=0$, $\forall\,\eta_l\in\Omega_N$.
Indeed, by the property of the Lagrange polynomial
$L_{\Omega_N}\left[e_{N+1}\right]$, we will have:
$\forall\,\eta_l\in\Omega_N$,
$L_{\Omega_N}\left[e_{N+1}\right]\left(\eta_l\right)
=e_{N+1}\left(\eta_l\right)
=e_{N+1}\left(\eta_l\right)-P_N\left(\eta_l\right)$.

Let be $\eta_l\in\Omega_N$, either $|l|\leq d$ or $|l|=d+1$.

If $|l|\leq d$, then either there is $j\in[\![1,m-2]\!]$ such that
$l_j\leq k_j-1$, either $l_{m-1}\leq k_{m-1}$ or $l_s\leq k_m-2$. Otherwise
we would have\\
$d\geq|l|\geq\sum_{j=1}^{m-2}l_j+l_{m-1}+l_s
\geq\sum_{j=1}^{s-2}k_j+(k_{m-1}+1)+(k_m-1)=|k(N)|=d+1$,
and that is impossible.
Necessarily, either there is $j\in[\![1,m-2]\!]$ such that
$l_j\leq k_j-1$ (then 
$\prod_{i_j=0}^{k_j-1}\left(\eta^{(j)}_{l_j}-\eta^{(j)}_{i_j}\right)=0$),
either $l_{m-1}\leq k_{m-1}$ (then 
$\prod_{i_{m-1}=0}^{k_{m-1}}\left(\eta^{(m-1)}_{l_{m-1}}-\eta^{(m-1)}_{i_{m-1}}\right)=0$)
or $l_s\leq k_m-2$ (then 
$\prod_{i_s=0}^{k_m-2}\left(\eta^{(s)}_{l_s}-\eta^{(s)}_{i_s}\right)=0$).
In any case $P_N\left(\eta_l\right)=0$.

If $|l|=d+1$, then if there is $j\in[\![1,m-2]\!]$ such that
$l_j\leq k_j-1$, we get
$\prod_{i_j=0}^{k_j-1}\left(\eta^{(j)}_{l_j}-\eta^{(j)}_{i_j}\right)=0$)
and $P_N\left(\eta_l\right)=0$. Otherwise
$l_j\geq k_j$ for all $j=1,\ldots,m-2$. In particular,
$l_1\geq k_1$ then $l_1=k_1$ since
$l\leq k(N)$ (by the definition of the lexicographic order
for $|l|=d+1=|k(N)|$, see~(\ref{lexicorder})).
It follows that $l_2\leq k_2$, i.e. $l_2=k_2$, and 
an immediate induction on $j=1,\ldots,m-2$ yields
$l_j=k_j$.\\
Once again, since $l\leq k(N)$, we necessarily 
have $l_{m-1}\leq k_{m-1}$ then 
$\prod_{i_{m-1}=0}^{k_{m-1}}\left(\eta^{(m-1)}_{l_{m-1}}-\eta^{(m-1)}_{i_{m-1}}\right)=0$)
and $P_N\left(\eta_l\right)=0$.

As a conclusion, both polynomials $L_{\Omega_N}[e_{N+1}](z)$
and $e_{N+1}(z)-P_N(z)$ belong to $\mathcal{P}_N$ and coincide on
$\Omega_N$ that is unisolvent for $\mathcal{P}_N$.
This proves equality~(\ref{explicitVDMiiaux}).

\medskip

$(b)$ If $m=s$, i.e.
$k(N)=\left(k_1,\ldots,k_s\right)$ with $k_s\geq1$,
we get
$k(N+1)=\left(k_1,\ldots,k_{s-2},k_{s-1}+1,k_s-1\right)$,
then 
\begin{eqnarray*}
\left\{\begin{array}{rcl}
H_{N+1}\;=\;
\eta_{k(N+1)}
& = &\left(\eta^{(1)}_{k_1},\ldots,\eta^{(s-2)}_{k_{s-2}},\eta^{(s-1)}_{k_{s-1}+1},
\eta^{(s)}_{k_s-1}\right)
\\
e_{N+1}(z) & = &
\left(\prod_{j=1}^{m-2}z_j^{k_j}\right)
\times
z_{s-1}^{k_{s-1}+1}z_s^{k_s-1}
\,.
\end{array}\right.
\end{eqnarray*}
It remains to prove that for all $z\in\C^s$,
\begin{eqnarray*}
L_{\Omega_N}[e_{N+1}](z)
& = &
e_{N+1}(z)
-
\prod_{j=1}^{s-2}\left[\prod_{i_j=0}^{k_j-1}\left(z_j-\eta^{(j)}_{i_j}\right)\right]
\times
\prod_{i_{s-1}=0}^{k_{s-1}}\left(z_{s-1}-\eta^{(s-1)}_{i_{s-1}}\right)
\times
\prod_{i_s=0}^{k_s-2}\left(z_s-\eta^{(s)}_{i_s}\right)
\,.
\end{eqnarray*}
The proof is similar as in the previous case~$(a)$. On the other hand, 
the associated expression of $P_N$ follows by replacing 
$m$ with $s$ in formula~(\ref{explicitVDMii}).

\end{proof}

\begin{remark}\label{OmegaNunisolvent}

By this way, we could also get back by induction on $N\geq1$
the proof that $\Omega_N$ is
unisolvent for $\mathcal{P}_N$. Indeed, the induction hypothesis
allows us to consider the Lagrange polynomial $L_{\Omega_N}[e_{N+1}]$
in order to obtain the usefull factorisation~(\ref{explicitVDMi})
(resp., ~(\ref{explicitVDMii})) from which we can deduce
(by replacing $z$ with $H_{N+1}$) that
$\vdm\left(H_1,\ldots,H_{N},H_{N+1}\right)\neq0$
(since every $\left(\eta^{(j)}_i\right)_{i\geq0}$ is a sequence
of pairwise distinct elements).
This finally proves that $\Omega_{N+1}$ is still unisolvent for
$\mathcal{P}_{N+1}$ and the induction is achieved.

\end{remark}

\subsection{Proofs of Theorem~\ref{explicitlejasequ} and Corollary~\ref{redemschifsic}}

Now we can give the proof of Theorem~\ref{explicitlejasequ}.

\begin{proof}

We begin with the first implication. Let be $\left(H_N\right)_{N\geq1}$
an intertwining and Leja sequence for $K=K_1\times\cdots\times K_s$
and let fix $j\in[\![1,s]\!]$. We prove by induction on $n\geq0$
that $\left(\eta^{(j)}_0,\ldots,\eta^{(j)}_n\right)$
is an $(n+1)$-Leja section.

First, we always have 
$\eta^{(j)}_0\in\partial K_j$ then
$\left(\eta^{(j)}_0\right)$ is a $1$-Leja section.

Next, let be $n\geq0$ such that
$\left(\eta^{(j)}_0,\ldots,\eta^{(j)}_n\right)$ is an $(n+1)$-Leja section.
Two cases must be studied.

If $1\leq j\leq s-1$, we consider 
$N$ such that $k(N)=(0,\ldots,n,1,0,\ldots,0)$,
where the index of $n$ is $j$
(such an $N\geq1$ exists and is unique).
Then $k(N+1)=(0,\ldots,0,n+1,0,\ldots,0)$ and
an application of~(\ref{explicitVDMii}) from Proposition~\ref{explicitVDM}
(with $m=j+1$, $k_j(N)=n$, $k_{j+1}(N)=1$ and $k_i(N)=0$ otherwise)
gives for all $z\in\C^s$:
\begin{eqnarray*}
\vdm\left(H_1,\ldots,H_N,z\right)
& = &
\prod_{i=0}^n\left(z_j-\eta^{(j)}_i\right)
\times\vdm\left(H_1,\ldots,H_N\right)
\,.
\end{eqnarray*}
It follows that 
$H_{N+1}=\eta_{k(N+1)}
=\left(\eta^{(1)}_0,\ldots,\eta^{(j-1)}_0,\eta^{(j)}_{n+1},\eta^{(j+1)}_0,\ldots,\eta^{(s)}_0\right)$ 
reaches\\
$\max_{z\in K_1\times\cdots\times K_s}\left|\vdm\left(H_1,\ldots,H_N,z\right)\right|$ 
if and only if
$\eta^{(j)}_{n+1}$ reaches
$\max_{z_j\in K_j}\left|\prod_{i=0}^n\left(z_j-\eta^{(j)}_i\right)\right|$.
Since $\left(\eta^{(j)}_0,\ldots,\eta^{(j)}_n\right)$ is an $(n+1)$-Leja section by
induction hypothesis, we deduce that\\
$\left(\eta^{(j)}_0,\ldots,\eta^{(j)}_n,\eta^{(j)}_{n+1}\right)$
is an $(n+2)$-Leja section as well.

Otherwise $j=s$ and we consider $N$ such that
$k(N)=(n,0,\ldots,0)$. Then\\
$k(N+1)=(0,\ldots,0,n+1)$ and an application of~(\ref{explicitVDMi}) from 
Proposition~\ref{explicitVDM}
leads to:
\begin{eqnarray*}
\forall\,z\in\C^s,
& &
\vdm\left(H_1,\ldots,H_N,z\right)
\;=\;
\prod_{i=0}^n\left(z_s-\eta^{(s)}_i\right)
\times\vdm\left(H_1,\ldots,H_N\right)
\,.
\end{eqnarray*}
Thus $H_{N+1}=\left(\eta^{(1)}_0,\ldots,\eta^{(s-1)}_0,\eta^{(s)}_{n+1}\right)$
reaches
$\max_{z\in K_1\times\cdots\times K_s}\left|\vdm\left(H_1,\ldots,H_N,z\right)\right|$ 
if and only if
$\eta^{(s)}_{n+1}$ reaches
$\max_{z_s\in K_s}\left|\prod_{i=0}^n\left(z_s-\eta^{(s)}_i\right)\right|$.
Since $\left(\eta^{(s)}_0,\ldots,\eta^{(s)}_n\right)$ is an $(n+1)$-Leja section by
induction hypothesis, it follows that
$\left(\eta^{(s)}_0,\ldots,\eta^{(s)}_n,\eta^{(s)}_{n+1}\right)$
is an $(n+2)$-Leja section as well.

The induction is achieved in all the cases for $j=1,\ldots,s$ \
and this proves the first implication.

\medskip

Now we deal with the reverse implication. We consider
$\left(H_N\right)_{N\geq1}$ the intertwining sequence
from the $s$ Leja sequences
$\left(\eta^{(j)}_{i_j}\right)_{i_j\geq0}$ that also satisfy
$\eta^{(j)}_0\in\partial K_j$ for all $j=1,\ldots,s$.
The proof is by induction on $N\geq1$. 

First,
we immediately get 
$H_1=\left(\eta^{(1)}_0,\ldots,\eta^{(s)}_0\right)
\in\prod_{j=1}^s\partial K_j\subset\partial\left(\prod_{j=1}^sK_j\right)$
then $\left(H_1\right)$ is a $1$-Leja section.

Next, let be $N\geq1$. An application of Proposition~\ref{explicitVDM}
with $z=H_{N+1}=\eta_{k(N+1)}$ yields
$\vdm\left(H_1,\ldots,H_N,H_{N+1}\right)
=P_N\left(H_{N+1}\right)\times\vdm\left(H_1,\ldots,H_N\right)$.
Without loss of generality, we can assume that $P_N$ is given by~(\ref{explicitVDMii}),
the other case being similar. 
We then have $k(N)=\left(k_1,\ldots,k_m,0,\ldots,0\right)$
with $2\leq m\leq s$ and $k_m\geq1$, and
$k(N+1)=\\
\left(k_1,\ldots,k_{m-2},k_{m-1}+1,0,\ldots,0,k_m-1\right)$
(if $m=s$, we just have to ignore the chain of $0$'s).
Since all the sequences
$\left(\eta^{(j)}_{i_j}\right)_{i_j\geq0}$ are one-dimensional Leja sequences, we can deduce that
for all $j=1,\ldots,m-2$ (resp., $j=m$),
we have\\
$\left|\prod_{i_j=0}^{k_j-1}\left(\eta^{(j)}_{k_j}-\eta^{(j)}_{i_j}\right)\right|=
\max_{z_j\in K_j}\left|\prod_{i_j=0}^{k_j-1}\left(z_j-\eta^{(j)}_{i_j}\right)\right|$
(resp., 
$\left|\prod_{i_s=0}^{k_m-2}\left(\eta^{(s)}_{k_m-1}-\eta^{(s)}_{i_s}\right)\right|
=\\
\max_{z_s\in K_s}\left|\prod_{i_s=0}^{k_m-2}\left(z_s-\eta^{(s)}_{i_s}\right)\right|$).
Notice that if $k_j=0$ (resp., $k_m=1$), 
the associated product is empty then equals $1$ whose maximum is reached on any
point, in particular on $\eta^{(j)}_0=\eta^{(j)}_{k_j}$
(resp., $\eta^{(s)}_0=\eta^{(s)}_{k_m-1}$).
Similarly, one has
$\left|\prod_{i_{m-1}=0}^{k_{m-1}}\left(\eta^{(m-1)}_{k_{m-1}+1}-\eta^{(m-1)}_{i_{m-1}}\right)\right|=\\
\max_{z_{m-1}\in K_{m-1}}\left|\prod_{i_{m-1}=0}^{k_{m-1}}\left(z_{m-1}-\eta^{(m-1)}_{i_{m-1}}\right)\right|$.
It follows that
\begin{eqnarray*}
\left|\vdm\left(H_1,\ldots,H_N,H_{N+1}\right)\right|
& = &
\qquad\qquad\qquad\qquad\qquad\qquad\qquad
\qquad\qquad\qquad\qquad\qquad\qquad\qquad
\end{eqnarray*}
\begin{eqnarray*}
& = &
\prod_{j=1}^{m-2}
\left[
\prod_{i_j=0}^{k_j-1}\left|\eta^{(j)}_{k_j}-\eta^{(j)}_{i_j}\right|
\right]
\prod_{i_{m-1}=0}^{k_{m-1}}\left|\eta^{(m-1)}_{k_{m-1}+1}-\eta^{(m-1)}_{i_{m-1}}\right|
\prod_{i_s=0}^{k_m-2}\left|\eta^{(s)}_{k_m-1}-\eta^{(s)}_{i_s}\right|
\times
\left|\vdm\left(H_1,\ldots,H_N\right)\right|
\\
& = &
\prod_{j=1}^{m-2}
\left[
\max_{z_j\in K_j}\prod_{i_j=0}^{k_j-1}\left|z_j-\eta^{(j)}_{i_j}\right|
\right]
\times
\max_{z_{m-1}\in K_{m-1}}\prod_{i_{m-1}=0}^{k_{m-1}}\left|z_{m-1}-\eta^{(m-1)}_{i_{m-1}}\right|
\times
\max_{z_s\in K_s}\prod_{i_s=0}^{k_m-2}\left|z_s-\eta^{(s)}_{i_s}\right|
\times
\\
& &
\qquad\qquad\qquad\qquad\qquad\qquad\qquad\qquad
\qquad\qquad\qquad\qquad\qquad\qquad\qquad
\times
\left|\vdm\left(H_1,\ldots,H_N\right)\right|
\\
& = &
\max_{z\in K_1\times\cdots\times K_s}\left|P_N(z)\right|
\times
\left|\vdm\left(H_1,\ldots,H_N\right)\right|
\\
& = &
\max_{z\in K_1\times\cdots\times K_s}
\left|\vdm\left(H_1,\ldots,H_N,z\right)\right|
\,,
\end{eqnarray*}
the last equality being another application of Proposition~\ref{explicitVDM}.

On the other hand, we know by the induction hypothesis that
$\left(H_1,\ldots,H_N\right)$ is an $N$-Leja section.
It follows that $\left(H_1,\ldots,H_N,H_{N+1}\right)$
is also an $(N+1)$-Leja section and the induction is achieved,
as well as the proof of the theorem.

\end{proof}

Now we can give the proof of Corollary~\ref{redemschifsic} given in the Introduction, that is the formula
for the two-variable Vandermonde determinant of Schiffer and Siciak.

\begin{proof}

Here we deal with a bidimensional intertwining sequence that we can
denote 
$\left(\eta_k,\theta_l\right)$ with the associated numeration~(\ref{numNs}) for 
$(k,l)\in\N^2$. The proof of the corollary is by induction on $d\geq0$.

For $d=0$, we have
$\vdm\left(\eta_0,\theta_0\right)
=\det\left(1\right)=1=\prod_{\emptyset}$\,.

Now let be $d\geq0$. Respective applications of~$(i)$ 
and~$(ii)$ from Proposition~\ref{explicitVDM} (with $m=s=2$) give:
\begin{enumerate}

\item[$(i)$]
if $N=N_d=\dfrac{(d+1)(d+2)}{2}$ with $d\geq0$, then
$H_N=\left(\eta_d,\theta_0\right)$,
$H_{N+1}=\left(\eta_0,\theta_{d+1}\right)$ and
\begin{eqnarray}\label{redemschifsici}
\vdm\left(H_1,\ldots,H_N,H_{N+1}\right)
& = &
\prod_{j=0}^d\left(\theta_{d+1}-\theta_j\right)
\times\vdm\left(H_1,\ldots,H_N\right)
\,;
\end{eqnarray}

\item[$(ii)$]
otherwise $N_d<N<N_{d+1}$ with $d\geq0$, then
$H_N=\left(\eta_k,\theta_l\right)$,
$H_{N+1}=\left(\eta_{k+1},\theta_{l-1}\right)$
with $k+l=d+1$, $l\geq1$ and
\begin{eqnarray}\label{redemschifsicii}
& &
\vdm\left(H_1,\ldots,H_N,H_{N+1}\right)
\;=\;
\prod_{i=0}^k\left(\eta_{k+1}-\eta_i\right)
\times
\prod_{j=0}^{l-2}\left(\theta_{l-1}-\theta_j\right)
\times
\vdm\left(H_1,\ldots,H_N\right)
\,.
\end{eqnarray}

\end{enumerate}

Successive applications of~(\ref{redemschifsicii}) 
(with $(k,l)=(d+1,0),(d,1),\ldots,(1,d)$ and 
$H_{N+1}=\left(\eta_{k},\theta_{l}\right)$) yield
\begin{eqnarray*}
\vdm
\left(
\left(\eta_0,\theta_0\right),\left(\eta_0,\theta_1\right),
\left(\eta_1,\theta_0\right),\ldots,\left(\eta_d,\theta_0\right),
\left(\eta_0,\theta_{d+1}\right),\ldots,\left(\eta_{d+1},\theta_0\right)
\right)
& = &
\;\;\;\;\;\;\;\;\;\;\;\;\;\;\;\;\;\;\;\;\;\;\;\;\;\;\;\;\;\;\;\;\;\;\;\;
\;\;\;\;\;\;\;\;\;\;\;\;\;\;\;\;\;\;\;\;\;\;\;\;\;\;\;\;\;\;\;\;\;\;\;\;
\end{eqnarray*}
\begin{eqnarray*}
& = &
\prod_{i=0}^{d}\left(\eta_{d+1}-\eta_i\right)
\times
\vdm
\left(
\left(\eta_0,\theta_0\right),\ldots,\left(\eta_d,\theta_0\right),
\left(\eta_0,\theta_{d+1}\right),\ldots,\left(\eta_d,\theta_1\right)
\right)
\\
& = &
\prod_{i=0}^{d}\left(\eta_{d+1}-\eta_i\right)
\prod_{i=0}^{d-1}\left(\eta_d-\eta_i\right)
\left(\theta_1-\theta_0\right)
\vdm
\left(
\left(\eta_0,\theta_0\right),\ldots,\left(\eta_d,\theta_0\right),
\left(\eta_0,\theta_{d+1}\right),\ldots,\left(\eta_{d-1},\theta_2\right)
\right)
\\
&  \vdots & 
\\
& = &
\prod_{m=0}^d
\left[
\prod_{i=0}^{m}\left(\eta_{m+1}-\eta_i\right)
\,
\prod_{j=0}^{d-m-1}\left(\theta_{d-m}-\theta_j\right)
\right]
\times
\vdm
\left(
\left(\eta_0,\theta_0\right),\ldots,\left(\eta_d,\theta_0\right),\left(\eta_0,\theta_{d+1}\right)
\right)
\\
& = &
\prod_{m=0}^d
\left[
\prod_{i=0}^{m}\left(\eta_{m+1}-\eta_i\right)
\prod_{j=0}^{d-m-1}\left(\theta_{d-m}-\theta_j\right))
\right]
\times
\prod_{j=0}^d
\left(\theta_{d+1}-\theta_j\right)
\times
\vdm
\left(
\left(\eta_0,\theta_0\right),\ldots,\left(\eta_d,\theta_0\right)
\right),
\end{eqnarray*}
the last equality being an application of~(\ref{redemschifsici}).
It follows that
\begin{eqnarray*}
\vdm
\left(
\left(\eta_0,\theta_0\right),\left(\eta_0,\theta_1\right),
\left(\eta_1,\theta_0\right),\ldots,\left(\eta_d,\theta_0\right),
\left(\eta_0,\theta_{d+1}\right),\ldots,\left(\eta_{d+1},\theta_0\right)
\right)
& = &
\;\;\;\;\;\;\;\;\;\;\;\;\;\;\;\;\;\;\;\;\;\;\;\;\;\;\;\;\;\;\;\;
\;\;\;\;\;\;\;\;\;\;\;\;\;\;\;\;\;\;\;\;\;\;\;\;\;\;\;\;\;\;\;\;\;\;\;\;
\end{eqnarray*}
\begin{eqnarray*}
& = &
\prod_{m=0}^d
\left[
\prod_{i=0}^{m}\left(\eta_{m+1}-\eta_i\right)
\,\times\,
\prod_{j=0}^{m}\left(\theta_{m+1}-\theta_j\right)
\right]
\times
\vdm
\left(
\left(\eta_0,\theta_0\right),\ldots,\left(\eta_d,\theta_0\right)
\right)
\,,
\end{eqnarray*}
because
$\prod_{m=0}^d\prod_{j=0}^{d-m-1}\left(\theta_{d-m}-\theta_j\right)=
\prod_{m=0}^{d-1}\prod_{j=0}^{d-m-1}\left(\theta_{d-m}-\theta_j\right)=
\prod_{m=0}^{d-1}\prod_{j=0}^m\left(\theta_{m+1}-\theta_j\right)$.
Since
\begin{eqnarray*}
\prod_{m=0}^d
\left[
\prod_{i=0}^{m}\left(\eta_{m+1}-\eta_i\right)
\right]
\;=\;
\prod_{0\leq i<j\leq d+1}
\left(\eta_j-\eta_i\right)
\;=\;
\vdm\left(\eta_0,\ldots,\eta_{d+1}\right)
\end{eqnarray*}
(similarly, 
$\prod_{m=0}^d
\left[\,\prod_{j=0}^{m}\left(\theta_{m+1}-\theta_j\right)\right]
=\vdm\left(\theta_0,\ldots,\theta_{d+1}\right)$),
we finally get
\begin{eqnarray*}
\vdm
\left(
\left(\eta_0,\theta_0\right),\left(\eta_0,\theta_1\right),
\left(\eta_1,\theta_0\right),\ldots,\left(\eta_d,\theta_0\right),
\left(\eta_0,\theta_{d+1}\right),\ldots,\left(\eta_{d+1},\theta_0\right)
\right)
& = &
\;\;\;\;\;\;\;\;\;\;\;\;\;\;\;\;\;\;\;\;\;\;\;\;\;\;\;\;\;\;\;\;
\;\;\;\;\;\;\;\;\;\;\;\;\;\;\;\;\;\;\;\;\;\;\;\;\;\;\;\;\;\;\;\;\;\;\;\;
\end{eqnarray*}
\begin{eqnarray*}
& = &
\vdm\left(\eta_0,\ldots,\eta_{d+1}\right)
\times
\vdm\left(\theta_0,\ldots,\theta_{d+1}\right)
\times
\vdm
\left(
\left(\eta_0,\theta_0\right),\ldots,\left(\eta_d,\theta_0\right)
\right)
\,,
\end{eqnarray*}
and the induction is achieved.

\end{proof}

\subsection{On the non-intertwining Leja sequences.}\label{Lejanonintertwining}

Here we consider the special case of $s=2$ and $K=\overline{\D}^2$.
We begin with setting
\begin{eqnarray*}
H_1\;=\;(1,1)\,,\quad
H_2\;=\;\left(-1,-1\right)\,,\quad
H_3\;=\;\left(e^{i\pi/4},-e^{i\pi/4}\right)
\end{eqnarray*}
and for all $n\geq3$, $H_n$ is constructed by induction
so that
\begin{eqnarray}\label{ctrexLejanonint2}
\left|
\vdm\left(H_1,\ldots,H_n,H_{n+1}\right)
\right|
& = &
\sup_{(z,w)\in\overline{\D}^2}
\left|
\vdm\left(H_1,\ldots,H_n,(z,w)\right)
\right|
\,.
\end{eqnarray}
Since 
\begin{eqnarray*}
\vdm\left(H_1,(z,w)\right)
\;=\;
\begin{array}{|cc|}
1 & 1\\
1 & w
\end{array}
\;=\;
w-1
\end{eqnarray*}
and 
\begin{eqnarray*}
\vdm\left(H_1,H_2,(z,w)\right)
\;=\;
\begin{array}{|ccc|}
1 & 1 & 1\\
1 & -1 & w\\
1 & -1 & z
\end{array}
\;=\;
2(w-z)
\;,
\end{eqnarray*}
we have
\begin{eqnarray*}
\left|\vdm\left(H_1,H_2\right)\right|
& = &
\sup_{(z,w)\in\overline{\D}^2}
\left|\vdm\left(H_1,(z,w)\right)\right|
\end{eqnarray*}
and
\begin{eqnarray*}
\left|\vdm\left(H_1,H_2,H_3\right)\right|
& = &
\sup_{(z,w)\in\overline{\D}^2}
\left|\vdm\left(H_1,H_2,(z,w)\right)\right|
\end{eqnarray*}
(as well as $H_1\in\partial K$).
This proves that $\left(H_1,H_2,H_3\right)$ is a $3$-Leja section.

On the other hand, equation~(\ref{ctrexLejanonint2}) is solvable
for all $n\geq3$ then the sequence $\left(H_n\right)_{n\geq1}$
is well-defined and is a Leja sequence for the compact subset
$\overline{\D}^2$ (although the other terms $H_n$ cannot be made
explicit for $n\geq4$).

Finally, $\left(H_n\right)_{n\geq1}$ cannot be written as an intertwining sequence.
Indeed, if it were the case, there would exists two sequences
$\left(\eta_i\right)_{i\geq0}$
and $\left(\theta_j\right)_{j\geq0}$ from which
$\left(H_n\right)_{n\geq1}$ would be the intertwining one.
In particular, we would have
$\left(\eta_0,\theta_0\right)=H_1=(1,1)$
and 
$\left(\eta_0,\theta_1\right)=H_2=(-1,-1)$
then $1=\eta_0=-1$
and that is impossible.

In addition, none of the components of
$\left(H_n\right)_{n\geq1}$ is a Leja sequence
because
$\left(1,-1,e^{i\pi/4}\right)$
(resp., $\left(1,-1,-e^{i\pi/4}\right)$)
is not a $3$-Leja section.

We can deduce the following result.

\begin{proposition}\label{ctrexLejanonint}

There exists a bidimensional Leja sequence that cannot be written
as an intertwining one. In addition, none of its components
is a (one-dimensional) Leja sequence.

\end{proposition}

\begin{remark}

Nevertheless, we have the following question: 
if $\left(H_n\right)_{n\geq1}$ is a (nultidimensional) Leja sequence
such that the shifted sequence
$\left(H_n\right)_{n\geq n_0}$ is an intertwining one (where
$n_0\geq2$), can it be written as an intertwining sequence of
(one-dimensional) Leja ones?

The answer may be negative since the shifted sequence of a Leja sequence
is not a Leja one any more (if we consider
$\left(\eta_i\right)_{i\geq0}$ from~(\ref{explicitleja}),
then $\left(\eta_i\right)_{i\geq1}$ is not a Leja sequence
since $\left(-1,i\right)$ is not a $2$-Leja section).

\end{remark}

\section{Some estimates for the Lagrange interpolation on
bidimensional intertwining Leja sequences}\label{appl2var}

\subsection{Explicit formulas for the Lagrange polynomials on intertwining sequences}

For any given $N\geq1$, we remind from Subsection~\ref{appl2varintro} in the Introduction
the numbers $d$ and $m$ with $0\leq m\leq d$, the space
$\mathcal{P}_N\subset\C_{d}[z,w]$
spanned by $\C_{d-1}[z,w]$ and the monomials
$w^d,zw^{d-1},\ldots,z^mw^{d-m}$, and the set
$\Omega_N=\left\{H_k\right\}_{1\leq k\leq N}
=
\left\{
\left(\eta_0,\theta_0\right),\ldots,
\left(\eta_{d-1},\theta_0\right),
\left(\eta_0,\theta_d\right),
\ldots,\left(\eta_m,\theta_{d-m}\right)
\right\}$ for fixed sequences
$\left(\eta_k\right)_{k\geq0}$ and $\left(\theta_l\right)_{l\geq0}$
of pairwise distinct elements (that are not necessarily Leja sequences throughout this subsection).
We give the proof of the following result that has been mentioned
in the Introduction and that gives an explicit formula for the bidimensional
fundamental Lagrange polynomials (FLIPs) associated with 
$\mathcal{P}_N$ and $\Omega_N$.

\begin{proposition}\label{explicitlagrangepol}

Let $N\geq1$ and consider the associated $n$, $m$, $\mathcal{P}_N$ and $\Omega_N$.
Then the multivariate fundamental Lagrange polynomials
$l_{H_k}^{(N)}$ exist, i.e. for all $k=1,\ldots,N$,
$\exists\;l_{H_k}^{(N)}\in\mathcal{P}_N$ that satisfies
$l_{H_k}^{(N)}\left(H_l\right)=\delta_{k,l}$
for all $l=1,\ldots,N$.
In addition, for $k=1,\ldots,N$ let
$p,\,q$ be such that $H_k=\left(\eta_p,\theta_q\right)$. We have
for all $(z,w)\in\C^2$:

\begin{itemize}

\item

if $p+q=d$, or $p+q=d-1$ and $p\geq m+1$,
then
\begin{eqnarray}\label{nn-1m+1}
l_{(\eta_p,\theta_q)}^{(N)}(z,w)
& = &
\prod_{i=0}^{p-1}\dfrac{z-\eta_i}{\eta_p-\eta_i}
\,\times\,
\prod_{j=0}^{q-1}\dfrac{w-\theta_j}{\theta_q-\theta_j}
\,;
\end{eqnarray}

\item

if $p+q=d-1$ and $p=m$, then 
\begin{eqnarray}\label{n-1m}
& &
l_{(\eta_m,\theta_{d-m-1})}^{(N)}(z,w)
\;=\;
\prod_{i=0}^{m-1}\dfrac{z-\eta_i}{\eta_{m}-\eta_i}
\,\times\,
\prod_{j=0,j\neq {d-m-1}}^{d-m}\dfrac{w-\theta_j}{\theta_{d-m-1}-\theta_j}
\,;
\end{eqnarray}

\item

if $p+q=d-1$ and $0\leq p\leq m-1$, 
then
\begin{eqnarray}\label{n-1m-1}
l_{(\eta_p,\theta_q)}^{(N)}(z,w)
& = &
\;\;\;\;\;\;\;\;\;\;\;\;\;\;\;\;\;\;\;\;\;\;\;\;\;\;\;\;\;\;\;\;\;\;\;\;
\;\;\;\;\;\;\;\;\;\;\;\;\;\;\;\;\;\;\;\;\;\;\;\;\;\;\;\;\;\;\;\;\;\;\;\;
\;\;\;\;\;\;\;\;\;\;\;\;\;\;\;\;\;\;\;\;\;\;\;\;\;\;\;\;\;\;\;\;\;\;\;\;
\end{eqnarray}
\begin{eqnarray*}
\;\;\;\;\;\;\;
& = &
\prod_{i=0,i\neq p}^{p+1}\dfrac{z-\eta_i}{\eta_p-\eta_i}
\,
\prod_{j=0}^{q-1}\dfrac{w-\theta_j}{\theta_q-\theta_j}
\,-\,
\prod_{i=0}^{p-1}\dfrac{z-\eta_i}{\eta_p-\eta_i}
\,
\prod_{j=0}^{q-1}\dfrac{w-\theta_j}{\theta_q-\theta_j}
\,+\,
\prod_{i=0}^{p-1}\dfrac{z-\eta_i}{\eta_p-\eta_i}
\,
\prod_{j=0,j\neq q}^{q+1}\dfrac{w-\theta_j}{\theta_q-\theta_j}
\,;
\end{eqnarray*}

\item

if $0 \leq p+q\leq d-2$, $0\leq p\leq m-1$ and $0\leq q\leq d-m-1$, then
\begin{eqnarray}\label{n-2m-1n-m-1}
l_{(\eta_p,\theta_q)}^{(N)}(z,w)
& = &
\;\;\;\;\;\;\;\;\;\;\;\;\;\;\;\;\;\;\;\;\;\;\;\;\;\;\;\;\;\;\;\;\;\;\;\;
\;\;\;\;\;\;\;\;\;\;\;\;\;\;\;\;\;\;\;\;\;\;\;\;\;\;\;\;\;\;\;\;\;\;\;\;
\;\;\;\;\;\;\;\;\;\;\;\;\;\;\;\;\;\;\;\;\;\;\;\;\;\;\;\;\;\;\;\;\;\;\;\;
\end{eqnarray}
\begin{eqnarray*}
& = &
\prod_{i=0}^{p-1}\dfrac{z-\eta_i}{\eta_p-\eta_i}
\prod_{j=0,j\neq q}^{d-p}\dfrac{w-\theta_j}{\theta_q-\theta_j}
-
\prod_{i=0}^{p-1}\dfrac{z-\eta_i}{\eta_p-\eta_i}
\prod_{j=0,j\neq q}^{d-p-1}\dfrac{w-\theta_j}{\theta_q-\theta_j}
+
\prod_{i=0,i\neq p}^{p+1}\dfrac{z-\eta_i}{\eta_p-\eta_i}
\prod_{j=0,j\neq q}^{d-p-1}\dfrac{w-\theta_j}{\theta_q-\theta_j}
\\
& &
+\;
\sum_{r=1}^{m-p-1}
\left[
\prod_{i=0,i\neq p}^{p+r+1}\dfrac{z-\eta_i}{\eta_p-\eta_i}
\prod_{j=0,j\neq q}^{d-p-r-1}\dfrac{w-\theta_j}{\theta_q-\theta_j}
\;-\;
\prod_{i=0,i\neq p}^{p+r}\dfrac{z-\eta_i}{\eta_p-\eta_i}
\prod_{j=0,j\neq q}^{d-p-r-1}\dfrac{w-\theta_j}{\theta_q-\theta_j}
\right]
\\
& &
+\;
\sum_{r=m-p}^{d-p-q-2}
\left[
\prod_{i=0,i\neq p}^{p+r+1}\dfrac{z-\eta_i}{\eta_p-\eta_i}
\prod_{j=0,j\neq q}^{d-p-r-2}\dfrac{w-\theta_j}{\theta_q-\theta_j}
\;-\;
\prod_{i=0,i\neq p}^{p+r}\dfrac{z-\eta_i}{\eta_p-\eta_i}
\prod_{j=0,j\neq q}^{d-p-r-2}\dfrac{w-\theta_j}{\theta_q-\theta_j}
\right]
\,;
\end{eqnarray*}

\item

if $0 \leq p+q\leq d-2$, $0\leq p\leq m-1$ and $q\geq d-m$, then
\begin{eqnarray}\label{n-2m-1n-m}
l_{(\eta_p,\theta_q)}^{(N)}(z,w)
& = &
\;\;\;\;\;\;\;\;\;\;\;\;\;\;\;\;\;\;\;\;\;\;\;\;\;\;\;\;\;\;\;\;\;\;\;\;
\;\;\;\;\;\;\;\;\;\;\;\;\;\;\;\;\;\;\;\;\;\;\;\;\;\;\;\;\;\;\;\;\;\;\;\;
\;\;\;\;\;\;\;\;\;\;\;\;\;\;\;\;\;\;\;\;\;\;\;\;\;\;\;\;\;\;\;\;\;\;\;\;
\end{eqnarray}
\begin{eqnarray*}
& = &
\prod_{i=0}^{p-1}\dfrac{z-\eta_i}{\eta_p-\eta_i}
\prod_{j=0,j\neq q}^{d-p}\dfrac{w-\theta_j}{\theta_q-\theta_j}
-
\prod_{i=0}^{p-1}\dfrac{z-\eta_i}{\eta_p-\eta_i}
\prod_{j=0,j\neq q}^{d-p-1}\dfrac{w-\theta_j}{\theta_q-\theta_j}
+
\prod_{i=0,i\neq p}^{p+1}\dfrac{z-\eta_i}{\eta_p-\eta_i}
\prod_{j=0,j\neq q}^{d-p-1}\dfrac{w-\theta_j}{\theta_q-\theta_j}
\\
& &
+\;
\sum_{r=1}^{d-p-q-1}
\left[
\prod_{i=0,i\neq p}^{p+r+1}\dfrac{z-\eta_i}{\eta_p-\eta_i}
\prod_{j=0,j\neq q}^{d-p-r-1}\dfrac{w-\theta_j}{\theta_q-\theta_j}
\;-\;
\prod_{i=0,i\neq p}^{p+r}\dfrac{z-\eta_i}{\eta_p-\eta_i}
\prod_{j=0,j\neq q}^{d-p-r-1}\dfrac{w-\theta_j}{\theta_q-\theta_j}
\right]
\,;
\end{eqnarray*}

\item
if $p+q\leq d-2$ and $p=m$, then 
\begin{eqnarray}\label{n-2m}
l_{(\eta_p,\theta_q)}^{(N)}(z,w)
\;=\;
l_{(\eta_m,\theta_q)}^{(N)}(z,w)
& = &
\;\;\;\;\;\;\;\;\;\;\;\;\;\;\;\;\;\;\;\;\;\;\;\;\;\;\;\;\;\;\;\;\;\;\;\;
\;\;\;\;\;\;\;\;\;\;\;\;\;\;\;\;\;\;\;\;\;\;\;\;\;\;\;\;\;\;\;\;\;\;\;\;
\;\;\;\;\;\;\;\;
\end{eqnarray}
\begin{eqnarray*}
& = &
\prod_{i=0}^{m-1}\dfrac{z-\eta_i}{\eta_m-\eta_i}
\prod_{j=0,j\neq q}^{d-m}\dfrac{w-\theta_j}{\theta_q-\theta_j}
-
\prod_{i=0}^{m-1}\dfrac{z-\eta_i}{\eta_m-\eta_i}
\prod_{j=0,j\neq q}^{d-m-2}\dfrac{w-\theta_j}{\theta_q-\theta_j}
+
\prod_{i=0,i\neq m}^{m+1}\dfrac{z-\eta_i}{\eta_m-\eta_i}
\prod_{j=0,j\neq q}^{d-m-2}\dfrac{w-\theta_j}{\theta_q-\theta_j}
\\
& &
+\;
\sum_{r=1}^{d-m-q-2}
\left[
\prod_{i=0,i\neq m}^{m+r+1}\dfrac{z-\eta_i}{\eta_m-\eta_i}
\prod_{j=0,j\neq q}^{d-m-r-2}\dfrac{w-\theta_j}{\theta_q-\theta_j}
\;-\;
\prod_{i=0,i\neq m}^{m+r}\dfrac{z-\eta_i}{\eta_m-\eta_i}
\prod_{j=0,j\neq q}^{d-m-r-2}\dfrac{w-\theta_j}{\theta_q-\theta_j}
\right]
\,;
\end{eqnarray*}

\item

if $p+q\leq d-2$ and $p\geq m+1$, then 
\begin{eqnarray}\label{n-2m+1}
l_{(\eta_p,\theta_q)}^{(N)}(z,w)
& = &
\;\;\;\;\;\;\;\;\;\;\;\;\;\;\;\;\;\;\;\;\;\;\;\;\;\;\;\;\;\;\;\;\;\;\;\;
\;\;\;\;\;\;\;\;\;\;\;\;\;\;\;\;\;\;\;\;\;\;\;\;\;\;\;\;\;\;\;\;\;\;\;\;
\;\;\;\;\;\;\;\;\;\;\;\;\;\;\;\;\;\;\;\;\;\;\;\;\;\;\;\;\;\;\;\;\;\;\;\;
\end{eqnarray}
\begin{eqnarray*}
& = &
\prod_{i=0}^{p-1}\dfrac{z-\eta_i}{\eta_p-\eta_i}
\prod_{j=0,j\neq q}^{d-p-1}\dfrac{w-\theta_j}{\theta_q-\theta_j}
-
\prod_{i=0}^{p-1}\dfrac{z-\eta_i}{\eta_p-\eta_i}
\prod_{j=0,j\neq q}^{d-p-2}\dfrac{w-\theta_j}{\theta_q-\theta_j}
+
\prod_{i=0,i\neq p}^{p+1}\dfrac{z-\eta_i}{\eta_p-\eta_i}
\prod_{j=0,j\neq q}^{d-p-2}\dfrac{w-\theta_j}{\theta_q-\theta_j}
\\
& &
+\;
\sum_{r=1}^{d-p-q-2}
\left[
\prod_{i=0,i\neq p}^{p+r+1}\dfrac{z-\eta_i}{\eta_p-\eta_i}
\prod_{j=0,j\neq q}^{d-p-r-2}\dfrac{w-\theta_j}{\theta_q-\theta_j}
\;-\;
\prod_{i=0,i\neq p}^{p+r}\dfrac{z-\eta_i}{\eta_p-\eta_i}
\prod_{j=0,j\neq q}^{d-p-r-2}\dfrac{w-\theta_j}{\theta_q-\theta_j}
\right]
\,.
\end{eqnarray*}

\end{itemize}

\end{proposition}

The proof of this proposition is an application of the following lemma.

\begin{lemma}\label{prexplicitlagrangepol}

For all $\left(\eta_p,\theta_q\right)\in\Omega_N$,
the function that appears in the claimed equality~(\ref{nn-1m+1})
(resp., (\ref{n-1m}), (\ref{n-1m-1}), (\ref{n-2m-1n-m-1}), (\ref{n-2m-1n-m}), (\ref{n-2m})
and~(\ref{n-2m+1})) satisfies the required properties for the FLIP $l_{(\eta_p,\theta_q)}^{(N)}$:
\begin{enumerate}

\item\label{prexplicitlagrangepol1}

$l_{(\eta_p,\theta_q)}^{(N)}
\;\in\;
\mathcal{P}_N$\,;

\item\label{prexplicitlagrangepol2}

$l_{(\eta_p,\theta_q)}^{(N)}\left(\eta_k,\theta_l\right)
\;=\;
\delta_{p,k}\,\delta_{q,l}
\;,\;
\forall\,\left(\eta_k,\theta_l\right)\in\Omega_N$.

\end{enumerate}

\end{lemma}

\begin{proof}

First, all the involved polynomials are well-defined since the $\eta_i$'s
(resp., $\theta_j$'s) are supposed to be pairwise distinct.
Next, we will only deal with formula~(\ref{n-2m-1n-m-1}) since the proofs
of ~(\ref{nn-1m+1}), (\ref{n-1m}), (\ref{n-1m-1}),
(\ref{n-2m-1n-m}), (\ref{n-2m}) and~(\ref{n-2m+1}) are similar (else easier).

Let $p$, $q$ be positive integers with $p+q\leq d-2$, $p\leq m-1$ and $q\leq d-m-1$.
We first want to prove that the involved polynomial belongs to
$\mathcal{P}_N$, i.e. it has total degree at most $d$, and the 
products whose total degree equals $d$ must have partial degree at most
$m$ with respect to $z$.
Here, all these products have total degree
at most $d$. Next, one has
$\deg_z\left(
\prod_{i=0}^{p-1}\dfrac{z-\eta_i}{\eta_p-\eta_i}
\prod_{j=0,j\neq q}^{d-p}\dfrac{w-\theta_j}{\theta_q-\theta_j}
\right)=p\leq m-1$.
Similarly, 
$\deg_z\left(
\prod_{i=0,i\neq p}^{p+1}\dfrac{z-\eta_i}{\eta_p-\eta_i}
\prod_{j=0,j\neq q}^{d-p-1}\dfrac{w-\theta_j}{\theta_q-\theta_j}
\right)=p+1\leq m$,
and
for all $r=1,\ldots,m-p-1$ (in case $m\geq p+2$),
$\deg_z
\left(
\prod_{i=0,i\neq p}^{p+r+1}\dfrac{z-\eta_i}{\eta_p-\eta_i}
\prod_{j=0,j\neq q}^{d-p-r-1}\dfrac{w-\theta_j}{\theta_q-\theta_j}
\right)
=p+r+1\leq m$.
All the remaining products have total degree at most $d-1$, then this proves
part~(\ref{prexplicitlagrangepol1}).

\medskip

Now we prove part~(\ref{prexplicitlagrangepol2}).
First, one has $d-p>d-p-1\geq q+1>q$. 
On the other hand,
for all $r=1,\ldots,m-p-1$ (in case $m-p\geq2$, otherwise the associated sum does not
even appear), one has 
$d-p-r-1\geq d-m\geq q+1>q$. Similarly, 
for all $r=m-p,\ldots,d-p-q-2$ (in case $q\leq d-m-2$), one has 
$d-p-r-2\geq q>q-1$. 
It follows that the expression in~(\ref{n-2m-1n-m-1}) is divisible by
$\prod_{j=0}^{q-1}\dfrac{w-\theta_j}{\theta_q-\theta_j}$.
We similarly check that it is also divisible by 
$\prod_{i=0}^{p-1}\dfrac{z-\eta_i}{\eta_p-\theta_i}$.
Then it cancels all the points
$\left(\eta_k,\theta_l\right)$ with $0\leq k\leq p-1$ or $0\leq l\leq q-1$.
Thus it is sufficient to check~(\ref{prexplicitlagrangepol2})
for all $\left(\eta_k,\theta_l\right)\in\Omega_N$ with
$k\geq p$ and $l\geq q$.

Next, if we fix $z=\eta_p$, the expression in~(\ref{n-2m-1n-m-1}) becomes
\begin{eqnarray*}
\prod_{j=0,j\neq q}^{d-p}\dfrac{w-\theta_j}{\theta_q-\theta_j}
-
\prod_{j=0,j\neq q}^{d-p-1}\dfrac{w-\theta_j}{\theta_q-\theta_j}
+
\prod_{j=0,j\neq q}^{d-p-1}\dfrac{w-\theta_j}{\theta_q-\theta_j}
\;\;\;\;\;\;\;\;\;\;\;\;\;\;\;\;\;\;\;\;\;\;\;\;\;\;\;\;\;\;\;\;\;\;\;\;
\;\;\;\;\;\;\;\;\;\;\;\;\;\;\;\;\;\;\;\;\;\;\;\;\;\;\;\;\;\;\;\;\;\;\;\;
\;\;\;\;\;\;\;
\\
+
\sum_{r=1}^{m-p-1}
\left[
\prod_{j=0,j\neq q}^{d-p-r-1}\dfrac{w-\theta_j}{\theta_q-\theta_j}
-
\prod_{j=0,j\neq q}^{d-p-r-1}\dfrac{w-\theta_j}{\theta_q-\theta_j}
\right]
+\;
\sum_{r=m-p}^{d-p-q-2}
\left[
\prod_{j=0,j\neq q}^{d-p-r-2}\dfrac{w-\theta_j}{\theta_q-\theta_j}
-
\prod_{j=0,j\neq q}^{d-p-r-2}\dfrac{w-\theta_j}{\theta_q-\theta_j}
\right]
\;=\;
\end{eqnarray*}
\begin{eqnarray*}
\;=\;
\prod_{j=0,j\neq q}^{d-p}\dfrac{w-\theta_j}{\theta_q-\theta_j}
+\sum_{r=1}^{m-p-1}(0)
+\sum_{r=m-p}^{d-p-q-2}(0)
\;=\;
\prod_{j=0,j\neq q}^{d-p}\dfrac{w-\theta_j}{\theta_q-\theta_j}
& &
\qquad\qquad\qquad\qquad\qquad\qquad\qquad
\end{eqnarray*}
(notice that the above equalities hold if $p=m-1$ or $q=d-m-1$).\\
We first get $1$ for $w=\theta_q$.
Next, if $w=\theta_l$ with $l\geq q+1$, since 
$p+l\leq d$, we have $q+1\leq l\leq d-p$ and the above product vanishes.

Now if we fix $w=\theta_q$ in~(\ref{n-2m-1n-m-1}),
this gives
\begin{eqnarray*}
& &
\prod_{i=0}^{p-1}\dfrac{z-\eta_i}{\eta_p-\eta_i}
-
\prod_{i=0}^{p-1}\dfrac{z-\eta_i}{\eta_p-\eta_i}
+
\prod_{i=0,i\neq p}^{p+1}\dfrac{z-\eta_i}{\eta_p-\eta_i}
\\
& &
+
\sum_{r=1}^{m-p-1}
\left[
\prod_{i=0,i\neq p}^{p+r+1}\dfrac{z-\eta_i}{\eta_p-\eta_i}
\;-\;
\prod_{i=0,i\neq p}^{p+r}\dfrac{z-\eta_i}{\eta_p-\eta_i}
\right]
+
\sum_{r=m-p}^{d-p-q-2}
\left[
\prod_{i=0,i\neq p}^{p+r+1}\dfrac{z-\eta_i}{\eta_p-\eta_i}
\;-\;
\prod_{i=0,i\neq p}^{p+r}\dfrac{z-\eta_i}{\eta_p-\eta_i}
\right]
\;=
\end{eqnarray*}
\begin{eqnarray*}
& = &
\prod_{i=0,i\neq p}^{p+1}\dfrac{z-\eta_i}{\eta_p-\eta_i}
+
\sum_{r=1}^{d-p-q-2}
\left[
\prod_{i=0,i\neq p}^{p+r+1}\dfrac{z-\eta_i}{\eta_p-\eta_i}
\;-\;
\prod_{i=0,i\neq p}^{p+r}\dfrac{z-\eta_i}{\eta_p-\eta_i}
\right]
\\
& = &
\prod_{i=0,i\neq p}^{p+1}\dfrac{z-\eta_i}{\eta_p-\eta_i}
+
\prod_{i=0,i\neq p}^{d-q-1}\dfrac{z-\eta_i}{\eta_p-\eta_i}
-
\prod_{i=0,i\neq p}^{p+1}\dfrac{z-\eta_i}{\eta_p-\eta_i}
\;=\;
\prod_{i=0,i\neq p}^{d-q-1}\dfrac{z-\eta_i}{\eta_p-\eta_i}
\end{eqnarray*}
(once again, the above equalities hold if $p=m-1$ or $q=d-m-1$).\\
We first get $1$ for $z=\eta_p$. 
Next, let fix $z=\eta_k$ with $k\geq p+1$. 
If $k+q\leq d-1$, then $p+1\leq k\leq d-q-1$
and the above product vanishes. Otherwise
$k+q=d$ then $k\leq m$ (since 
$\left(\eta_k,\theta_q\right)$ must belong to $\Omega_N$)
and $q=d-k\geq d-m$. This is incompatible with the condition that
$q\leq d-m-1$.

The remaining case is the one for which
$(z,w)=\left(\eta_k,\theta_l\right)\in\Omega_{N}$ with
$k\geq p+1$ and $l\geq q+1$.
Since $k+l\leq d$, one has $l\leq d-k\leq d-p-1$ (with $l\geq q+1$),
then 
$\prod_{j=0,j\neq q}^{d-p-1}\dfrac{\theta_l-\theta_j}{\theta_q-\theta_j}=0$
and the first three terms in~(\ref{n-2m-1n-m-1}) vanish.
Next, let be $r$ with $1\leq r\leq m-p-1$ (we can assume 
$m-p\geq2$, otherwise the first sum disappears). If
$l\leq d-p-r-1$, then
$\prod_{j=0,j\neq q}^{d-p-r-1}\dfrac{\theta_l-\theta_j}{\theta_q-\theta_j}=0$
(since $l\geq q+1$).
Otherwise $l\geq d-p-r$ then $k\leq d-l\leq p+r$ (with $k\geq p+1$),
and 
$\prod_{i=0,i\neq p}^{p+r}\dfrac{\eta_k-\eta_i}{\eta_p-\eta_i}=0$.
It follows that the first sum in~(\ref{n-2m-1n-m-1}) always vanishes.\\
Lastly, let be $r$ with $m-p\leq r\leq d-p-q-2$ (similarly, we can assume that
$d-q-2\geq m$).
If $l\leq d-p-r-2$ (with $l\geq q+1$), then
$\prod_{j=0,j\neq q}^{d-p-r-2}\dfrac{\theta_l-\theta_j}{\theta_q-\theta_j}=0$.
Otherwise $l\geq d-p-r-1$. If $k+l\leq d-1$, then
$k\leq d-l-1\leq p+r$ (with $k\geq p+1$) and
$\prod_{i=0,i\neq p}^{p+r}\dfrac{\eta_k-\eta_i}{\eta_p-\eta_i}=0$;
otherwise $k+l=d$ then necessarily $k\leq m\leq p+r$
(since $r\geq m-p$) and one still has
$\prod_{i=0,i\neq p}^{p+r}\dfrac{\eta_k-\eta_i}{\eta_p-\eta_i}=0$.
It follows that the second sum in~(\ref{n-2m-1n-m-1}) vanishes and this completes
the proof of~(\ref{n-2m-1n-m-1}).

\end{proof}

\begin{remark}\label{OmegaNunisolvent2}

These formulas could also have been deduced by using the results 
from~\cite{sauerxu1995} and~\cite{calvi2005}
where the authors give algorithms to construct them in the general case of
{\em block unisolvent arrays} in $\C^d$. Here we independently computed
them for the special case of $\Omega_N$ and $\mathcal{P}_N$ in $\C^2$.

On the other hand, we can get back from the previous lemma the following result whose proof is a
classical reasoning (and that can be also deduced from Section~\ref{applmultivar},
Proposition~\ref{explicitVDM} for the case $s=2$, see Remark~\ref{OmegaNunisolvent}): 
for all $N\geq1$,
the set $\Omega_N$ is unisolvent for $\mathcal{P}_N$. As an equivalent consequence,
for every function $f$ that is defined on $\Omega_N$, there exists a unique
polynomial $P\in\mathcal{P}_N$ such that
$P(z,w)=f(z,w)$ for every $(z,w)\in\Omega_N$.
In addition, $P$ is the multivariate Lagrange polynomial
$L_{\Omega_N}[f](z,w)
=
\sum_{n=1}^N
f\left(H_n\right)\,l_{H_n}^{(N)}(z,w)$,
where the $l_{H_n}^{(N)}$'s are the polynomials mentioned in the statement of
Proposition~\ref{explicitlagrangepol}.

\end{remark}

As an application, it follows that these polynomials $l_{H_n}^{(N)}$'s are indeed the 
required FLIPs associated with $\mathcal{P}_N$ and $\Omega_N$
(since $\Omega_N$ is unisolvent for $\mathcal{P}_N$) and this completes
the proof of Proposition~\ref{explicitlagrangepol}.

\subsection{Uniform estimates for the FLIPs on intertwining Leja sequences and applications}

First, we remind~\cite[Theorem~1.2]{irigoyen5}.

\begin{theorem}\label{unifestim}

Let $\mathcal{L}=\left(\eta_0,\eta_1,\eta_2,\ldots\right)$
be a Leja sequence for the unit disk (with $\left|\eta_0\right|=1$). Then
the FLIPs $l_k^{(N+1)}$ are uniformly
bounded with respect to $N\geq0$ and
$k=0,\ldots,N$, i.e.
\begin{eqnarray*}\label{estimate}
\sup_{N\geq k\geq0}
\left(
\sup_{z\in\overline{\D}}
\left|
l_k^{(N+1)}(z)\right|
\right)
& \leq &
\pi\exp(3\pi)
\,.
\end{eqnarray*}

\end{theorem}

We also remind~\cite[Theorem~1.3]{irigoyen5}
that gives an estimate of
the FLIPs for compact sets whose boundary is an Alper-smooth Jordan curve.

\begin{theorem}\label{unifestimcompact}

Let $\mathcal{L}=\left(\eta_i\right)_{i\geq0}$ be a Leja sequence for the unit disk with
$\left|\eta_0\right|=1$, and let 
$\Phi(\mathcal{L})=\left(\Phi\left(\eta_i\right)\right)_{i\geq0}$
be its image by 
the conformal mapping $\Phi$. We have for all $N\geq0$, 
\begin{eqnarray*}
\max_{0\leq p\leq N}
\left[
\sup_{z\in K}
\left|
\prod_{i=0,i\neq p}^N
\dfrac{z-\Phi(\eta_i)}{\Phi(\eta_p)-\Phi(\eta_i)}
\right|
\right]
& = &
O\left(N^{2A/\ln(2)}\right)
\,,
\end{eqnarray*}
where $A$ is a positive constant depending only on $K$.

\end{theorem}

Now we can 
give the proof of Proposition~\ref{unifestimintertwining} (whose immediate application
will be Theorem~\ref{lebcteintertwin}).

\begin{proof}

We begin with the proof of the first estimate. Since
$\left(\eta_i\right)_{i\geq0}$ and $\left(\theta_j\right)_{j\geq0}$
are Leja sequences for the unit disk (with
$\left|\eta_0\right|=\left|\theta_0\right|=1$), we get by an application of Theorem~\ref{unifestim} that
for all $k,l\geq0$ and $p,q$ with $0\leq p\leq k$ and $0\leq q\leq l$,
\begin{eqnarray*}
\sup_{z\in\overline{\D}}
\left|
\prod_{i=0,i\neq p}^k
\dfrac{z-\eta_i}{\eta_p-\eta_i}
\right|
\;\leq\;
\pi\exp(3\pi)
%
& \mbox{ and } &
\sup_{w\in\overline{\D}}
\left|
\prod_{j=0,j\neq q}^l
\dfrac{w-\theta_j}{\theta_q-\theta_j}
\right|
\;\leq\;
\pi\exp(3\pi)
\,.
\end{eqnarray*}
Consider $N\geq1$ and the associated numbers $d\geq0$ and $m$ with $0\leq m\leq d$.
For all $\left(\eta_p,\theta_q\right)\in\Omega_{N}$,
an application of formula~(\ref{nn-1m+1}) from Proposition~\ref{explicitlagrangepol}
(resp.,~(\ref{n-1m}), (\ref{n-1m-1}), (\ref{n-2m-1n-m-1}), (\ref{n-2m-1n-m}), (\ref{n-2m})
and~(\ref{n-2m+1})) yields the following estimate:
\begin{eqnarray*}
\sup_{(z,w)\in\overline{\D}^2}
\left|
l_{(\eta_p,\theta_q)}^{(N)}(z,w)
\right|
& \leq &
2\left(d-p-q+1\right)
\pi\exp(3\pi)\times\pi\exp(3\pi)
\,.
\end{eqnarray*}
The first part of the proposition follows 
since  by~(\ref{Nn}), $N\sim N_d\sim d^2/2$.

The proof of the second part is similar since Proposition~\ref{explicitlagrangepol} is still valid
with the points $\eta_i$'s and  $\theta_j$'s replaced with the $\Phi_1\left(\eta_i\right)$'s
and $\Phi_2\left(\theta_j\right)$'s respectively (notice that the points
$\Phi_1\left(\eta_i\right)$'s and $\Phi_2\left(\theta_j\right)$'s are well-defined
since all the $\eta_i$'s and  $\theta_j$'s belong to the unit circle).
The only difference is an application of Theorem~\ref{unifestimcompact}
instead of Theorem~\ref{unifestim} that yields
for all $k,l\geq0$ and $p,q$ with $0\leq p\leq k$ and $0\leq q\leq l$:
$\sup_{z\in K_1}
\left|
\prod_{i=0,i\neq p}^k
\dfrac{z-\Phi_1\left(\eta_i\right)}{\Phi_1\left(\eta_p\right)-\Phi_1\left(\eta_i\right)}
\right|
\leq
M_1(k+2)^{2A_1/\ln(2)}$
and
$\sup_{w\in K_2}
\left|
\prod_{j=0,j\neq q}^l
\dfrac{w-\Phi_2\left(\theta_j\right)}{\Phi_2\left(\theta_q\right)-\Phi_2\left(\theta_j\right)}
\right|
\leq
M_2(l+2)^{2A_2/\ln(2)}$, 
where $M_1$, $A_1$ and $M_2$, $A_2$ are positive constants depending only
on $K_1$ and $K_2$ respectively.
By repeating the same argument of the first part, we get for all
$\left(\eta_p,\theta_q\right)\in\Omega_{N}$:
\begin{eqnarray*}
\sup_{(z,w)\in K_1\times K_2}
\left|
l_{\left(\Phi_1\left(\eta_p\right),\Phi_2\left(\theta_q\right)\right)}^{(N)}(z,w)
\right|
\leq
2M_1M_2\left(d-p-q+1\right)
\max_{k+l\leq d}
\left[
(k+2)^{2A_1/\ln(2)}
(l+2)^{2A_2/\ln(2)}
\right]
\,.
\end{eqnarray*}
By setting $A=\max\left(A_1,A_2\right)$ and noticing that
$
\max_{k+l\leq d}(k+2)(l+2)
=\max_{k+l=d}(k+2)(l+2)
\leq\sup_{x\in[0,d]}(x+2)(d+2-x)
=\left(\dfrac{d}{2}+2\right)\left(d+2-\dfrac{d}{2}\right)
=
\left(\dfrac{d+4}{2}\right)^2$,
we get
\begin{eqnarray*}
\sup_{(z,w)\in K_1\times K_2}
\left|
l_{\left(\Phi_1\left(\eta_p\right),\Phi_2\left(\theta_q\right)\right)}^{(N)}(z,w)
\right|
& \leq &
\dfrac{2M_1M_2}{2^{4A/\ln(2)}}
(d-p-q+1)(d+4)^{4A/\ln(2)}
\,.
\end{eqnarray*}
This proves the second part of the 
proposition since by~(\ref{Nn}), $N\sim N_d\sim d^2/2$.

\end{proof}

\begin{remark}

In the statement of Theorem~\ref{lebcteintertwin}, we could get a precise bound 
for $O\left(N^{3/2}\right)$ in~(\ref{lebcteintertwinbidisc}) 
and for $O\left(N^{2A/\ln(2)+3/2}\right)$  in~(\ref{lebcteintertwincompact}) respectively,
by computing the sum 
$\sum_{\left(\eta_p,\theta_q\right)\in\Omega_N}\left(d-p-q+1\right)$ 
instead of crudely estimating it by $d\times N$. But it seems useless since 
it will not change the exponent of $d$ (or $N$) that may not be optimal.

\end{remark}

Now we deal with another application of Theorem~\ref{lebcteintertwin}
that is Corollary~\ref{appljackson}. We first remind the notations from~\cite{renwang}:
the Lipschitz space $Lip_{\gamma}\left(\overline{\D}^s\right)$, $0<\gamma\leq1$,
consists of all holomorphic functions 
$f\in\mathcal{O}\left(\D^s\right)\bigcap C\left(\overline{\D}^s\right)$ satisfying
$\left|f\left(e^{ih}\zeta\right)-f(\zeta)\right|
\leq
L\,|h|^{\gamma}$,
for any $\zeta\in\mathbb{T}^s$ (the unit torus in $\C^s$) and $h\in\R$
($L>0$ being the Lipschitz constant);
then $m\in\N$ being given,
$f\in\mathcal{O}\left(\D^s\right)\bigcap C\left(\overline{\D}^s\right)$ is said to belong to
$Lip_{\gamma}^m\left(\overline{\D}^s\right)$
if $\dfrac{\partial^{\alpha}f}{\partial z^{\alpha}}\in Lip_{\gamma}\left(\overline{\D}^s\right)$
for any $|\alpha|=\alpha_1+\cdots+\alpha_s\leq m$;
lastly, for all $d\geq0$, $\mathcal{P}_d^{(s)}$ is the space
of polynomials of total degree at most $d$
($\mathcal{P}_d^{(s)}$ means
$\mathcal{P}_N$ with
$N=\dbinom{s+d}{s}$, that is also $\C_d[z_1,\ldots,z_s]$,
see~(\ref{defPN}), (\ref{defNd}) and~(\ref{defOmegaNPNple})).
Next, we remind the reader the following result
given as Theorem~7.2 from~\cite{renwang} and that is a generalization of
Jackson's theorem in the polydisc.

\begin{theorem}\label{thm72renwang}

If $f\in Lip_{\gamma}^m\left(\overline{\D}^s\right)$ then for all $d\in\N$,
\begin{eqnarray*}
\inf_{P\in\mathcal{P}^{(s)}_d}
\left[
\sup_{z\in\overline{\D}^s}
\left|
f(z)-P(z)
\right|
\right]
& \leq &
L\dfrac{C_{m,\gamma}}{d^{m+\gamma}}
\,,
\end{eqnarray*}
where $C_{m,\gamma}$ is a positive constant depending only on $m$ and $\gamma$.

\end{theorem}

We can then give the proof of Corollary~\ref{appljackson}.

\begin{proof}

Let fix $s=2$, $m\geq3$, $0<\gamma\leq1$ and 
$f\in Lip_{\gamma}^m\left(\overline{\D}^2\right)$.
By Theorem~\ref{thm72renwang}, there is a positive constant
$M>0$ such that for all $d\geq1$, there exists $P_d\in\C_d[z,w]$ that satisfies
\begin{eqnarray}\label{thm72renwang2}
\sup_{(z,w)\in\overline{\D}^2}
\left|
f(z,w)-P_d(z,w)
\right|
& \leq &
\frac{M}{d^{m+\gamma}}
\,.
\end{eqnarray}

On the other hand, let $\left(\eta_i\right)_{i\geq0}$ and $\left(\theta_j\right)_{j\geq0}$
be Leja sequences for the unit disk (with $\left|\eta_0\right|=\left|\theta_0\right|=1$),
and let us consider the intertwining sequence $\left(H_n\right)_{n\geq1}$ defined
as in~(\ref{defintertwin}). For all
$N\geq1$ with 
$N_d<N\leq N_{d+1}$ (see~(\ref{Nn})),
let us consider $\Omega_N$ defined as 
in~(\ref{defOmegaN2}).

Now let $L_{\Omega_N}[f]$ be the Lagrange polynomial of $f$
defined by~(\ref{deflagpold}).
We claim that $L_{\Omega_N}\left[P_d\right]=P_d$. Indeed, 
$L_{\Omega_N}\left[P_d\right]\in\mathcal{P}_N$
and $P_d\in\C_d[z,w]\subset\mathcal{P}_N$.
Moreover, $L_{\Omega_N}\left[P_d\right]$ and $P_d$ coincide on
$\Omega_N$ that is unisolvent for $\mathcal{P}_N$.
This proves the claim.

Thus a classical calculation yields for all $(z,w)\in\overline{\D}^2$,
\begin{eqnarray}\nonumber
\left|
f(z,w)-L_{\Omega_N}[f](z,w)
\right|
& \leq &
\left|
f(z,w)-P_d(z,w)
\right|
+
\left|
P_d(z,w)-L_{\Omega_N}[f](z,w)
\right|
\\\label{appljacksonaux1}
& \leq &
\frac{M}{d^{m+\gamma}}
+
\left|
L_{\Omega_N}\left[P_d-f\right](z,w)
\right|
\,,
\end{eqnarray}
the second inequality being an application of~(\ref{thm72renwang2}).
On the other hand, since the Lebesgue constant is also the operator norm
of the linear operator $L_{\Omega_N}$ that is
the projection from $C\left(\overline{\D}^2\right)$ onto $\mathcal{P}_N$,
it follows that
\begin{eqnarray*}
\left|
L_{\Omega_N}\left[P_d-f\right](z,w)
\right|
& \leq &
\Lambda_N\left(\overline{\D}^2,\Omega_N\right)
\sup_{(z',w')\in\overline{\D}^2}
\left|
P_d\left(z',w'\right)-f\left(z',w'\right)
\right|
\\
& \leq &
\Lambda_N\left(\overline{\D}^2,\Omega_N\right)
\frac{M}{d^{m+\gamma}}
\,,
\end{eqnarray*}
the last estimate being justified by~(\ref{thm72renwang2}).
We deduce by~(\ref{appljacksonaux1}) that
\begin{eqnarray}\nonumber
\left|
f(z,w)-L_{\Omega_N}[f](z,w)
\right|
& \leq &
\left(
1+
\Lambda_N\left(\overline{\D}^2,\Omega_N\right)
\right)
\frac{M}{d^{m+\gamma}}
\,.
\end{eqnarray}
Finally, an application of~(\ref{lebcteintertwinbidisc}) from Theorem~\ref{lebcteintertwin}
leads to (since $N\sim d^2/2$ by~(\ref{Nn}) and~(\ref{defnm}))
\begin{eqnarray*}
\left|
f(z,w)-L_{\Omega_N}[f](z,w)
\right|
\;=\;
\left(1+
O\left(N^{3/2}\right)
\right)
\times
O\left(\frac{1}{\left(\sqrt{N}\right)^{m+\gamma}}\right)
\;=\;
O\left(\frac{1}{N^{(m+\gamma-3)/2}}\right).
\end{eqnarray*}

\end{proof}

\end{document}